%% file: main_arxiv.tex
\documentclass{article} % For LaTeX2e
\usepackage{amssymb}
\usepackage[a4paper, total={6in, 8in}]{geometry}
\usepackage{bm}
\usepackage{natbib}

\input{math_commands.tex}

\usepackage{hyperref}
\usepackage{url}
\usepackage{mathrsfs}
\usepackage{enumitem}
\usepackage{booktabs}
\usepackage{graphicx}
\usepackage{float}
\usepackage{mathtools}
\usepackage[bottom]{footmisc}
\usepackage{multirow}
\usepackage{amsthm}
\theoremstyle{plain}          % en-tete gras, corps en italique
\newtheorem{theorem}{Theorem}
\newtheorem{lemma}[theorem]{Lemma}
\newtheorem{proposition}[theorem]{Proposition}
\newtheorem{corollary}[theorem]{Corollary}

\theoremstyle{definition}     % en-tete gras, corps droit
\newtheorem{definition}[theorem]{Definition}

\newtheorem{assumption}[theorem]{Assumption}

\theoremstyle{remark}         % en-tete italique, corps droit
\newtheorem{remark}[theorem]{Remark}

\newcommand{\norm}[1]{\lVert#1\rVert}
\newcommand{\ip}[2]{\langle #1, #2\rangle}
\newcommand{\Cost}{\mathrm{Cost}}

\title{Convex Optimization Is Free When Accuracy Is Expensive}

\author{Arthur Paing\\
\'Ecole polytechnique\\
Palaiseau, France\\
\texttt{arthur.paing@polytechnique.edu} \\
\and
Arthur Jacot \\
Courant Institute, NYU \\
New York, USA \\
\texttt{arthur.jacot@nyu.edu}
}

\begin{document}

\maketitle

\begin{abstract}
This paper studies convex optimization when the gradient cannot be evaluated exactly, but only
approximated by a hierarchy of algorithms whose compute grows like $\delta^{-\gamma}$ in the
accuracy $\delta$. When $\gamma>2$, falling into the Harder-Than-Monte-Carlo (HTMC) regime, the price of accuracy outruns the variance reduction that Monte Carlo would buy and we show that minimizing a loss function costs no more, up to a factor depending only on $\gamma$, than
a single evaluation of its gradient at the accuracy the problem demands. A randomized multilevel oracle replaces the deterministic approximation of accuracy $\delta$ by an unbiased estimator of it, whose variance $\sigma^2$ becomes a second, independently priced dial: the cost of one call drops from $\delta^{-\gamma}$ to $\delta^{2-\gamma}\sigma^{-2}$. Plain inexact gradient descent driven by that oracle reaches loss
$\varepsilon$ at expected compute $\Theta(\varepsilon^{-\gamma})$ in the convex
case, against $\Theta(\varepsilon^{-(\gamma+1)})$ for the same method run at a fixed
accuracy: randomization buys a full power of $\varepsilon$. Under $\mu$-strong convexity the
exponent halves, to $\varepsilon^{-\gamma/2}$, because the iterates settle at a
noise floor and the bias budget relaxes accordingly. Both bounds are independent of the step
size, and hence of the smoothness constant, and we show that the cost is a functional of the underlying
gradient flow rather than of any discretization of it.

\end{abstract}

\section{Introduction}
Many optimization problems have a gradient that cannot be
evaluated exactly, only approximated at a cost that grows with the accuracy required: discretized
PDEs, nested simulations, and neural networks trained to approximate a score or a drift all behave
this way. We study convex optimization under such an oracle, in the regime where the cost grows fast
enough so that minimizing
the function costs no more than evaluating its gradient once, with the appropriate oracle. While convex optimization and exact gradient descent have been thoroughly studied, inexact gradient descent in the HTMC regime calls for an adaptation of the existing theory. Randomized multilevel estimators for biased oracles have been studied by
\citet{huwangchenhe2024}, who obtain the same cost exponents in the regime that corresponds
to $\gamma>2$, under assumptions bearing on a hierarchy of approximating \emph{functions}.
Accelerated methods under a jointly biased and noisy oracle are analyzed by
\citet{dvurechensky2016}, and the choice of an inexactness schedule when accuracy has a
price by \citet{vandesselglineur2023}. In both, the oracle exposes a single accuracy knob.
What is new here is an oracle with two, a bias and a variance bought at unrelated
exponents, and the consequence that the iteration count leaves the bill altogether.
Appendix~\ref{app:related} compares the settings in detail.

\section{Setup}

Throughout the rest of this paper, $\mathscr{L}:\R^d\to\R$ is our convex and $\beta$-smooth loss function with a minimizer $y^\ast$, and we
write $\mathscr{L}^\ast=\mathscr{L}(y^\ast)$, $f=\nabla \mathscr{L}$ and $d_0=\norm{y_0-y^\ast}$. In the strongly convex
statements $\mathscr{L}$ is moreover $\mu$-strongly convex with $0<\mu\le\beta$, and $\kappa=\beta/\mu$.

The gradient $f$ is not available exactly. What is available is a sequence of algorithms of
increasing accuracy and increasing compute.

\subsection{The approximate gradient oracle}

\begin{assumption}[Dyadic hierarchy]\label{ass:dyadic}
There is a family of algorithms $A_k$ approximating $f$ within $2^{-k}$, at a
compute growing exponentially in $k$:
\[
\norm{A_k-f}_\infty\le2^{-k},
\qquad
\Cost(A_k)\le c^\gamma2^{\gamma k}.
\]
\end{assumption}

Nothing that follows uses more than Assumption~\ref{ass:dyadic}, so every bound
applies verbatim to any hierarchy of that shape, with $c$ the prefactor of the
hierarchy at hand.

As it stands the assumption offers a single dial: running $A_k$ buys accuracy
$\delta=2^{-k}$ at cost $(c/\delta)^\gamma$. Randomizing which $A_k$ is run turns
that one dial into two, a bias and a variance, priced at unrelated exponents.

What it buys is shown in Table~\ref{tab:bounds}:

\begin{table}[h]
\begin{center}
\setlength{\tabcolsep}{6pt}
\begin{tabular}{llll}
\hline
& \multicolumn{1}{c}{\boldmath $\gamma$} & \multicolumn{1}{c}{\bf CONVEX}
  & \multicolumn{1}{c}{\bf STRONGLY CONVEX}\\
\hline
{\bf Deterministic}
 & Any $\gamma$    & $d_0^{\gamma+2}\,\beta\,\varepsilon^{-(\gamma+1)}$
                   & $(\mu\varepsilon)^{-\frac\gamma2}\,\kappa$\\
\hline
\multirow{3}{*}{\bf GD}
 & ETMC $\gamma<2$ & $d_0^2\,\beta^{1-\frac\gamma2}\,\varepsilon^{-(1+\frac\gamma2)}$
                   & $(\mu\varepsilon)^{-\frac\gamma2}\,\kappa^{1-\frac\gamma2}$\\
 & MC $\gamma=2$   & $d_0^2\,\varepsilon^{-2}\log^2\frac{\beta d_0^2}\varepsilon$
                   & $(\mu\varepsilon)^{-1}\log^2\kappa$\\
 & {\boldmath\bf HTMC $\gamma>2$}
                   & {\boldmath $d_0^\gamma\,\varepsilon^{-\gamma}$}
                   & {\boldmath $(\mu\varepsilon)^{-\frac\gamma2}$}\\
\hline
\multirow{3}{*}{\bf Accelerated}
 & {\boldmath\bf ETMC $\gamma<2$}
                   & {\boldmath $d_0^{1+\frac\gamma2}\beta^{\frac{2-\gamma}4}
                     \varepsilon^{-\frac{2+3\gamma}4}$}
                   & {\boldmath $(\mu\varepsilon)^{-\frac\gamma2}\,\kappa^{\frac{2-\gamma}4}$}\\
 & MC $\gamma=2$   & $d_0^2\,\varepsilon^{-2}\log^2\frac{\beta d_0^2}\varepsilon$
                   & $(\mu\varepsilon)^{-1}\log^2\kappa$\\
 & HTMC $\gamma>2$ & $d_0^{2\gamma-2}\beta^{\frac{\gamma-2}2}
                     \varepsilon^{-\frac{3\gamma-2}2}$
                   & $(\mu\varepsilon)^{-\frac\gamma2}\,\kappa^{\frac{\gamma-2}4}$\\
\hline
\end{tabular}
\caption{Expected compute to reach loss $\varepsilon$, up to constants depending
only on $\gamma$ and with a factor $c^\gamma$ omitted from every entry. The strongly
convex entries carry a further factor
$1+\log(\mu d_0^2/\varepsilon)$. The first two rows run the same iteration and differ only in the oracle driving it, a single $A_k$ held at a fixed level or the randomized telescope of Proposition~\ref{prop:oracle} and the third runs the accelerated scheme on that telescope. Bold marks the smallest bound of each regime, which is
acceleration below the threshold and plain gradient descent above it. }
\label{tab:bounds}
\end{center}
\end{table}

\begin{remark}[Where the hierarchy comes from]\label{rem:norms}
Assumption~\ref{ass:dyadic} is a statement about the HTMC norm of
\citet{jacot2025norms}. Writing $C(f,\varepsilon)$ for the least number of nodes of a
binary circuit approximating $f$ to within $\varepsilon$ in $L^\infty$ over a bounded
hyper-rectangle,
\begin{equation}\label{eq:htmc}
\norm{f}_{M^\gamma}^\gamma=\max_{\varepsilon>0}\ \varepsilon^{\gamma}\,C(f,\varepsilon)
\end{equation}
is the smallest constant for which $C(f,\varepsilon)\le\norm{f}^\gamma_{M^\gamma}
\varepsilon^{-\gamma}$ holds at every accuracy, and reading it at $\varepsilon=2^{-k}$
returns the $A_k$. Assumption~\ref{ass:dyadic} is therefore equivalent to
$\norm{\nabla\mathscr{L}}_{M^\gamma}\le c$ (up to a factor $2$ on $c$ from the dyadic
rounding) which is Assumption~1 of \citet{jacot2026mlem}. Since the bounds carry $c^\gamma$, they are tightest at the smallest admissible $c$,
which is the norm itself.

\end{remark}

 The threshold has a meaning
there as well: $\gamma>2$, \emph{Harder than Monte Carlo}, is where
$\norm{\cdot}_{M^\gamma}$ becomes convex up to a constant
\citep[Thm.~4]{jacot2025norms}, and measured rates sit far inside it,
$\gamma\approx18$ for language models \citep{kaplan2020} and $8$ to $15$ for natural image denoising
\citep{henighan2020}.

\subsection{Upgrading to a weakly biased oracle through randomization}

\begin{proposition}[Upgrading the oracle]\label{prop:oracle}
Let Assumption~\ref{ass:dyadic} hold for some $\gamma>0$, fix $\delta=2^{-k_{\max}}$ and
$\sigma>0$, and let
\begin{equation}\label{eq:mlmc}
\tilde f_{\delta,\sigma}=\sum_{k=k_{\min}}^{k_{\max}}\frac{B_k}{p_k}\big(A_k-A_{k-1}\big),
\qquad
p_k=\min\Big\{C\,2^{-(1+\frac\gamma2)k},\,1\Big\},
\end{equation}
with $B_k\sim\mathrm{Bernoulli}(p_k)$ independent, drawn at each call. Its bias is at most
$\delta$ whatever the $p_k$, and $C$ may be chosen so that the variance is at most
$\sigma^2$, at expected compute
\[
\E\big[\Cost(\tilde f_{\delta,\sigma})\big]\ \asymp\
\begin{cases}
c^\gamma\,\delta^{2-\gamma}\,\sigma^{-2}, & \gamma>2\quad\text{(Harder-Than-Monte-Carlo)},\\[2pt]
c^2\,\sigma^{-2}\log^2\frac\sigma\delta, & \gamma=2\quad\text{(Monte-Carlo)},\\[2pt]
(c/\sigma)^\gamma, & \gamma<2\quad\text{(Easier-Than-Monte-Carlo)},
\end{cases}
\]
up to a factor depending only on $\gamma$, and no other choice of $p_k\in(0,1]$ does
better.
\end{proposition}

\textit{Sketch of proof.} The expectation of \eqref{eq:mlmc} telescopes to
$A_{k_{\max}}$, since $A_{k_{\min}-1}=0$ and each increment is reweighted by $p_k^{-1}$,
so the bias is $\norm{A_{k_{\max}}-f}\le\delta$ whatever the $p_k$. Everything else
follows from the constraint $p_k\le1$. The optimal $p_k\propto2^{-(1+\gamma/2)k}$
decreases in $k$, so it binds on an initial segment, up to some accuracy $\delta_\ast$.
There the Bernoulli weights are all $1$ and the telescope collapses,
$\sum_{k\le k_\ast}(A_k-A_{k-1})=A_{k_\ast}$: a \emph{single deterministic evaluation} at
accuracy $\delta_\ast$, of compute $(c/\delta_\ast)^\gamma$ and no variance. Optimizing
over where to place the cut,
\begin{equation}\label{eq:split}
\Cost(\delta,\sigma)\ \asymp\ \min_{\delta\le\delta_\ast\le c}
\Big\{\underbrace{\Big(\frac{c}{\delta_\ast}\Big)^{\!\gamma}}_{\text{deterministic prefix}}
+\underbrace{\frac{c^\gamma}{\sigma^2}\,\Sigma^2}_{\text{randomized tail}}\Big\},
\qquad
\Sigma=\!\!\sum_{\delta\le2^{-k}\le\delta_\ast}\!\!2^{(\frac\gamma2-1)k},
\end{equation}
and whether the ratio $2^{\gamma/2-1}$ exceeds one is the entire story. For $\gamma>2$ the
sum is dominated by its finest term, $\Sigma\asymp\delta^{1-\gamma/2}$: the tail does not
depend on $\delta_\ast$ at all while the prefix decreases in it, so the minimum is
attained with no prefix. For $\gamma<2$ it is dominated by its coarsest term instead,
$\Sigma\asymp\delta_\ast^{1-\gamma/2}$ and both terms then move with $\delta_\ast$ in
opposite directions and balance at $\delta_\ast=\sigma$. At $\gamma=2$ the ratio is one
and $\Sigma$ merely counts its terms. Details in Appendix~\ref{app:oracle}. \qed

\paragraph{Optimality of the multilevel oracle} We now show that this method allows us to upgrade from a biased oracle to an oracle that is the least biased possible in the ETMC and HTMC regimes and almost so in the MC regime. The cost of a random estimator with rate $\gamma$ should be of the form 
\[
C(\delta,\sigma)\leq c^{\gamma}\delta^{-\gamma}h\left(\frac{\delta}{\sigma}\right)
\]
for some increasing function $h(r)$, which represents the discount on the cost we get from increasing the variance at a fixed bias, and the faster $h(r)$ goes to $0$ as $r\searrow0$, the stronger this discount. We know that it cannot decay faster than $r^{\gamma}$ (otherwise if we fixed $\sigma$ and decreased $\delta$ the cost would decay to $0$), and it also cannot decay faster than $r^{2}$ (otherwise we could choose a growing variance $n\sigma_{0}^{2}$ and average it $n$ times to obtain a variance of $\sigma_{0}^{2}$ at a cost $n\delta^{-\gamma}h(\frac{\delta}{\sqrt{n}\sigma})$ which vanishes as $n\to\infty$). This implies that the ``least biased oracle'' would have $h(r)=r^{\min\{\gamma,2\}}$ which would be unbiased in the ETMC and MC regimes, and weakly biased in the HTMC regime. Our multilevel oracle matches this optimality in the ETMC and HTMC regimes, but not quite in the MC regime where we obtain $h(r)=r^{2}\log^{2}\frac{1}{r}$ instead of $h(r)=r^{2}$.

\paragraph{What the three branches say.} The exponents $2-\gamma$ and $-2$ of the HTMC
branch are unrelated, so bias and variance are bought separately. Since $2-\gamma<0$ a
\emph{cruder} estimator is a cheaper one, and the whole art of what follows consists in
running as biased an oracle as the geometry permits. Below the threshold that separation
collapses: the tail of \eqref{eq:split} no longer involves $\delta$, so bias becomes free
and the estimator may be taken unbiased, at the price of a single deterministic evaluation
at the tolerated noise level. Multilevel unbiases the oracle there but no longer at the rate $\sigma^{-2}$ that cancels the iteration count, and the advantage over the baseline of Table~\ref{tab:bounds} comes
from computing at that noise level rather than at the accuracy the target demands while
above the threshold it is the telescope that produces it.
At the threshold itself the bias is free only up to a $\log^2(\sigma/\delta)$. This is the one feature the
surrounding literature does not have: oracles whose accuracy can be bought at a price
appear in \citet{vandesselglineur2023}, and biased oracles with noise in
\citet{dvurechensky2016}, but in both the accuracy is a single scalar, and
\citet{vandesselglineur2023} name random inexactness of this kind as an open direction.

\section{Multilevel stochastic gradient descent}
\label{gen_inst}

We can now run plain inexact gradient descent with this oracle,
\begin{equation}\label{eq:iter}
y_{t+1}=y_t-\eta\,\tilde f_{\delta,\sigma}(y_t),\qquad t=0,\dots,T-1,
\end{equation}
with independent draws across steps, and measure the total cost as $T$ times the expected
compute of one call.

\subsection{Optimization costs one gradient evaluation}

\begin{theorem}[Cost of optimization]\label{thm:main}
Let $\mathscr{L}$ be convex and $\beta$-smooth under Assumption~\ref{ass:dyadic} with
rate $\gamma>0$. For every $\varepsilon>0$, and with a step size
\begin{equation}\label{eq:eta}
\eta\ \le\ \min\Big\{\frac{1}{4\beta},\ \frac{d_0^2}{\varepsilon}\Big\}\ \text{ when }\gamma>2,
\qquad\qquad
\eta\ =\ \frac{1}{4\beta}\ \text{ when }\gamma\le2,
\end{equation}
there are hyperparameters $T,\delta,\sigma$, none of which depends on $\gamma$, such that
the iteration \eqref{eq:iter} satisfies the following.
\begin{enumerate}[label=(\alph*),leftmargin=2.2em,itemsep=4pt]
\item \textbf{Convex.} With $T=\lceil\frac{4d_0^2}{\eta\varepsilon}\rceil$,
$\delta=\frac{\varepsilon}{16d_0}$ and $\sigma^2=\frac{\varepsilon}{4\eta}$, the
averaged iterate $\bar y_T=\frac1T\sum_{t<T}y_t$ satisfies
$\E[\mathscr{L}(\bar y_T)-\mathscr{L}^\ast]\le\varepsilon$ at expected cost
\[
\E[\Cost]\ \le\ \Lambda_\gamma\cdot
\begin{cases}
\big(c\,d_0/\varepsilon\big)^{\gamma}, & \gamma>2,\\[2pt]
c^2d_0^2\,\varepsilon^{-2}\log^2\frac{\beta d_0^2}\varepsilon, & \gamma=2,\\[2pt]
c^\gamma d_0^2\,\beta^{1-\frac\gamma2}\,\varepsilon^{-(1+\frac\gamma2)}, & \gamma<2 .
\end{cases}
\]
\item \textbf{Strongly convex.} If $\mathscr{L}$ is moreover $\mu$-strongly convex and
$\varepsilon\le\mu d_0^2$\footnote{Part~(b) is stated under $\varepsilon\le\mu d_0^2$ because that is exactly where its bound
improves on part~(a). Above that accuracy the convex bound is the smaller of the two, and
even a strongly convex objective is governed by $\varepsilon^{-\gamma}$ rather than
$\varepsilon^{-\gamma/2}$: which of the two binds is a property of the accuracy asked for,
not of the objective. Strong convexity also relates loss and distance by
$\varepsilon\asymp\mu\Delta^2$, so above the threshold the bound reads as a guarantee on
the iterate: one reaches $\E\norm{\hat y_T-y^\ast}^2\le\Delta^2$ at cost
$\widetilde O\big((c/(\mu\Delta))^{\gamma}\big)$.}, then with \linebreak
$T=\lceil\frac{4}{\eta\mu}\log\frac{256\mu d_0^2}{\varepsilon}\rceil$,
$\delta=\frac{\sqrt{\mu\varepsilon}}{8}$ and $\sigma^2=\frac{\varepsilon}{8\eta}$, the
geometrically weighted average \linebreak $\hat y_T\propto\sum_{t<T}(1-\eta\mu)^{-t}y_t$ satisfies
$\E[\mathscr{L}(\hat y_T)-\mathscr{L}^\ast]\le\varepsilon$ at expected cost
\[
\E[\Cost]\ \le\ \Lambda'_\gamma\Big(1+\log\frac{\mu d_0^2}\varepsilon\Big)\cdot
\begin{cases}
\big(c/\sqrt{\mu\varepsilon}\big)^{\gamma}, & \gamma>2,\\[2pt]
c^2(\mu\varepsilon)^{-1}\log^2\kappa, & \gamma=2,\\[2pt]
c^\gamma(\mu\varepsilon)^{-\frac\gamma2}\,\kappa^{1-\frac\gamma2}, & \gamma<2 .
\end{cases}
\]
\end{enumerate}
Here $\Lambda_\gamma$ and $\Lambda'_\gamma$ depend only on $\gamma$, and the three branches
do not meet at the threshold (Appendix~\ref{app:etmc}).\\
By contrast, running
\eqref{eq:iter} with the deterministic algorithm $A_{k_{\max}}$ costs
$\Theta\big(\beta d_0^2\varepsilon^{-1}\cdot(cd_0/\varepsilon)^\gamma\big)
=\Theta(\varepsilon^{-(\gamma+1)})$ in the convex case.

\end{theorem}

\textit{Sketch of proof.}
We follow the classical route. Write $a_t^2=\E\norm{y_t-y^\ast}^2$ and expand
$\norm{y_{t+1}-y^\ast}^2=\norm{y_t-y^\ast-\eta\tilde f_{\delta,\sigma}(y_t)}^2$,
conditioning on $y_t$ so that the centered part of the oracle drops out. The three
remaining terms are handled by $\mu$-strong convexity, with $\mu=0$ in the merely
convex case, by Cauchy--Schwarz for the bias, and by
$\norm{\nabla\mathscr{L}}^2\le2\beta(\mathscr{L}-\mathscr{L}^\ast)$ for the
smoothness. With $\eta\le\frac1{4\beta}$ this yields the master recursion
\begin{equation}\label{eq:master}
a_{t+1}^2\ \le\ (1-\eta\mu)\,a_t^2-\eta\,\E\big[\mathscr{L}(y_t)-\mathscr{L}^\ast\big]
+2\eta\delta\,a_t+\eta^2\tilde\sigma^2,
\qquad \tilde\sigma^2=\sigma^2+2\delta^2 .
\end{equation}
Both cases start from \eqref{eq:master} and differ only in how it is summed.

\paragraph{Convex case.} Set $\mu=0$ and sum over $t<T$: the squared distances
telescope to at most $d_0^2$, while a separate induction (Appendix~\ref{app:main})
shows that the iterates never leave the ball $a_t\le2d_0$, which is what controls
the bias term. Dividing by $\eta T$ and applying Jensen to $\bar y_T$ leaves three
terms,
\[
\E\big[\mathscr{L}(\bar y_T)-\mathscr{L}^\ast\big]\ \le\
\underbrace{\frac{d_0^2}{\eta T}}_{\text{horizon}}+\underbrace{4\delta d_0}_{\text{bias}}
+\underbrace{\eta\tilde\sigma^2}_{\text{noise}},
\]
and the announced $T,\delta,\sigma$ are the split of $\varepsilon$ between them.
The cost is $T$ times the per-call compute of Proposition~\ref{prop:oracle}, and above the threshold the two powers of $\eta$ cancel: $T\propto\eta^{-1}$ against $\sigma^{-2}\propto\eta$.

\paragraph{Strongly convex case.} Weighting \eqref{eq:master} by $(1-\eta\mu)^{-(t+1)}$
concentrates the weights on the phase where the iterates have settled within a \emph{noise
floor} $V\asymp\delta/\mu+\sqrt{\eta\tilde\sigma^2/\mu}$ of the optimum. The bias is then
measured against $V$ and not $d_0$, which relaxes its budget to $\sqrt{\mu\varepsilon}$ and
halves the exponent. Appendix~\ref{app:main} gives the details.

The same proof gives the other two rows of Table~\ref{tab:bounds}: only the per-call
compute of Proposition~\ref{prop:oracle} changes, and below the threshold it no longer
cancels against the horizon.

\medskip
\begin{corollary}[The step size]\label{cor:eta}
Since $T\propto\eta^{-1}$ and the tolerable variance $\sigma^2\propto\eta^{-1}$, the total
compute carries $T\sigma^{-s}\propto\eta^{s/2-1}$: \textbf{the step size cancels if and
only if $s=2$}, that is, if and only if $\gamma>2$. Below the threshold the exponent is
negative, so one takes $\eta$ maximal and the smoothness constant returns.
\end{corollary}

\begin{center}
\setlength{\tabcolsep}{6pt}\begin{tabular}{lccc}
\multicolumn{1}{c}{\bf REGIME} &\multicolumn{1}{c}{\bf $s$}
&\multicolumn{1}{c}{\bf FACTOR IN $\eta$} &\multicolumn{1}{c}{\bf $\eta$ USED}
\\ \hline
ETMC, $\gamma<2$ & $\gamma$ & $\eta^{\gamma/2-1}$ & $\frac1{4\beta}$, maximal\\
MC, $\gamma=2$   & $2$      & $\log^2(1/\eta)$    & $\frac1{4\beta}$\\
HTMC, $\gamma>2$ & $2$      & none                & free\\
\end{tabular}
\end{center}

The limits in Table~\ref{tab:bounds} are also
consistent at both ends:
\begin{itemize}[leftmargin=1.5em,itemsep=2pt,topsep=3pt]
\item \emph{as $\gamma\to0$}, the prefix swallows the hierarchy and the bounds
reduce to $\beta d_0^2/\varepsilon$ and $\kappa\log(1/\varepsilon)$, the
complexities of exact gradient descent
\item \emph{as $\gamma\to2$}, the convex exponent $1+\frac\gamma2$ meets $\gamma$
from the other side. The bounds themselves do not meet: the $\log^2$ at the
threshold is what the constant $C_\gamma\asymp(\gamma-2)^{-2}$ of
Proposition~\ref{prop:oracle} becomes at finite accuracy.
\end{itemize}

\noindent
The two regimes call for qualitatively different methods. Below the threshold the compute
grows as $\eta$ shrinks, so one wants the largest stable step: a classical stochastic
gradient descent. Above it the cancellation is exact rather than an artifact of the
constants (halving the step doubles the number of iterations and halves the price of
each) so $\beta$, which enters only through $\eta\le\frac1{4\beta}$, is absent from
those bounds: \textbf{the cost is a functional of the underlying gradient flow rather than
of any discretization of it}. 

There is then no reason to take $\eta$ maximal, as one would
with an exact gradient: a smaller step is free. The same phenomenon occurs for the
multilevel Euler--Maruyama method, whose $\eta\searrow0$ limit is a Poisson jump process
\citep{jacot2026mlem}. Nothing prevents taking that limit here either: the iteration
becomes a flow, $\delta$ and $\sigma$ become functions of time, and scheduling them
removes the logarithm. 

A consequence of the above, is that in the HTMC regime our complexity bounds are independent of the smoothness parameter $\beta$. Actually, the smoothness assumption might be unnecessary in this regime\footnote{Proving this would require describing the convergence of our algorithm to a multilevel Poisson jump process as $\eta\searrow 0$, and adapting the proofs to subgradients.}. This means that in the HTMC regime, our method can be applied without changes to LASSO problems \cite{tibshirani1996_LASSO} or other non-smooth losses, in contrast to the setting of exact oracles where special optimization algorithms \cite{efron2004_LARS} are needed to solve such problems.

\bigskip
\begin{proposition}[Time-varying accuracy]\label{prop:schedule}
Let $\mathscr{L}$ be $\mu$-strongly convex and $\beta$-smooth under Assumption~\ref{ass:dyadic}, and
let $\varepsilon>0$. Consider the small-step limit of \eqref{eq:iter}, run over a flow time
$T_c\ge\frac2\mu\log\frac{\mu d_0^2}{\varepsilon}$ with the time-varying accuracies
\[
\delta(t)=\sqrt{\frac{\mu\varepsilon}{8}}\ e^{\mu(T_c-t)/8},
\qquad
\eta\,\sigma^2(t)=\frac\varepsilon4\ e^{\mu(T_c-t)/4} .
\]
Then the weighted average $\hat y\propto\int_0^{T_c}e^{\mu t/2}y_t\,dt$ satisfies
$\E[F(\hat y)-F^\ast]\le\varepsilon$ at expected compute
\[
\E[\Cost]\ \le\ \Lambda''_\gamma\Big(\frac{c}{\sqrt{\mu\varepsilon}}\Big)^{\gamma},
\]
with no logarithmic factor and independently of the horizon $T_c$.
\end{proposition}

Schedules of this kind are studied by \citet{vandesselglineur2023} for a deterministic
oracle whose cost grows as $\delta^{-r}$, ours included, with the optimal sequence of
accuracies in closed form. What is scheduled here is the \emph{pair}
$(\delta(t),\sigma(t))$ of a randomized oracle, and the gain is of another nature: not a
better constant, but the disappearance of the logarithm.

An error committed early is contracted away before the
trajectory ends and need not be bought at full price, so the compute rate decays
geometrically backwards in time and only the last stretch is paid for. The horizon is still $\log$-long, but the cost is not: the compute integral converges
whatever the value of $T_c$. The proof
(Appendix~\ref{app:schedule}) is carried out in continuous time only, which is why
Theorem~\ref{thm:main}(b) still carries the logarithm.

\subsection{The cost exponent is optimal}

The bounds above are attained and they are also unimprovable. We read this in the
oracle model attached to Assumption~\ref{ass:dyadic}: an algorithm queries a point
and a level, receives $A_k(y)$, pays $c^\gamma2^{\gamma k}$ for it, and
computation between queries is free.

\begin{theorem}[Optimality]\label{thm:lower}
Let $\gamma,c,d_0>0$. Any algorithm that returns $\hat y$ with
$\E[\mathscr{L}(\hat y)-\mathscr{L}^\ast]\le\varepsilon$ on every instance of that
model with $\norm{y_0-y^\ast}\le d_0$ has, on some such instance,
$\E[\Cost]\ge\frac34 8^{-\gamma}(cd_0/\varepsilon)^{\gamma}$ when
$\varepsilon\le\beta d_0^2/8$ and, if it is required to succeed on $\mu$-strongly
convex instances with $\varepsilon\le\mu d_0^2/8$, then
$\E[\Cost]\ge\frac34(c/\sqrt{8\mu\varepsilon})^{\gamma}$.
\end{theorem}

\looseness=-1
Writing $\mathcal{I}_\gamma(d_0)$ for the instances of rate $\gamma$ started within $d_0$
of a minimizer, Theorems~\ref{thm:main} and~\ref{thm:lower} together read
{\setlength{\abovedisplayskip}{4pt}\setlength{\belowdisplayskip}{4pt}%
\setlength{\abovedisplayshortskip}{4pt}\setlength{\belowdisplayshortskip}{4pt}
\[
\min_{\text{algorithms}}\ \max_{\mathcal{I}_\gamma(d_0)}\ \E[\Cost]
\ \asymp\ \Big(\frac{c\,d_0}{\varepsilon}\Big)^{\gamma} ,
\]}
up to a factor depending only on $\gamma$, and $(c/\sqrt{\mu\varepsilon})^\gamma$ under
strong convexity, up to the logarithm. The statement is about the worst instance of the class and not
about every instance: many are easier, and the design of Section~\ref{sec:lsq} is one of
them. In the convex case randomization
moves the exponent: from the
$\varepsilon^{-(\gamma+1)}$ of the deterministic baseline to $\varepsilon^{-\gamma}$, and
Theorem~\ref{thm:lower} says that this full power of $\varepsilon$ is all there was to
take. In the strongly convex case it moves nothing: the baseline already runs at
$\gamma/2$, and $\gamma/2$ is optimal. What randomization buys there is the conditioning,
the baseline paying a factor $\kappa$ that we do not. The proof restricts the iteration in no way, so accelerated methods are covered as
well.

\looseness=-1
\textit{Sketch of proof.} Take any $\mathscr{L}_0$ in the class and flatten one direction:
$\mathscr{L}(s,z)=\mathscr{L}_0(z)+\frac\lambda2s^2$, then tilt,
$\mathscr{L}_\pm=\mathscr{L}\pm\delta s$. The two copies share a hierarchy at every level
coarser than $\delta$, so no algorithm separates them there and since $\mathscr{L}$ is
$\lambda$-flat along $s$, their optimality gaps sum to at least $\delta^2/\lambda$ at
every point, so they admit no common $\varepsilon$-minimizer. The algorithm must
therefore query below that resolution, at a cost of $(c/\delta)^\gamma$. Taking
$\lambda=8\varepsilon/d_0^2$ gives the convex bound, $\lambda=\mu$ the strongly convex one.

\begin{remark}[The exponent is also forced structurally]\label{rem:legendre}
Theorem~\ref{thm:lower} fixes an adversarial hierarchy. Under strong convexity the
exponent is forced by the framework alone, with no construction at all. Consider the Legendre transform
$\nabla\mathscr{L}^\ast(z)=\arg\min_y\{\mathscr{L}(y)-\ip zy\}$, a solver for
$\mathscr{L}$ \emph{is} an approximation scheme for $\nabla\mathscr{L}^\ast$, whose
scaling law is the solver's own cost profile. A method beating the exponent would
hand $\nabla\mathscr{L}^\ast$ a cheaper hierarchy, and conjugating back (the
class is stable and $\mathscr{L}^{\ast\ast}=\mathscr{L}$) would hand
$\nabla\mathscr{L}$ one cheaper than its own circuit complexity permits. At the
intrinsic rate
$\gamma_{\min}(f)=\inf\{\gamma:\norm{f}_{M^\gamma}<\infty\}$, optimizing therefore
costs exactly what evaluating the gradient once costs. Proofs are in
Appendix~\ref{app:lower}.
\end{remark}

\subsection{A word on acceleration}
\label{sec:nesterov}

\looseness=-1
In smooth convex optimization, acceleration trades iterations for precision: it reaches
accuracy $\varepsilon$ in $\varepsilon^{-1/2}$ steps instead of $\varepsilon^{-1}$, and
since each step still costs one gradient evaluation, the trade is free. Under
Assumption~\ref{ass:dyadic} it is not, because accuracy has a price: a sharper gradient is
a more expensive one, and whether the trade still pays depends on $\gamma$. Driving the
standard accelerated scheme with the same oracle and balancing its three error channels
against the cost model prices the trade exactly (Appendix~\ref{app:acc}).

\begin{proposition}[Convex]\label{prop:acc-cvx}
Accelerating multiplies the bound of Theorem~\ref{thm:main}(a) by $T^{\gamma-2}$ above the
threshold and by $T^{(\gamma-2)/2}$ below it, where $T\asymp d_0\sqrt{\beta/\varepsilon}$
is the accelerated horizon.
\end{proposition}

\begin{proposition}[Strongly convex]\label{prop:acc-sc}
Accelerating multiplies the bound of Theorem~\ref{thm:main}(b) by $\kappa^{\frac{\gamma-2}4}$,
in all three regimes.
\end{proposition}

\looseness=-1
Every exponent carries the sign of $\gamma-2$, and Theorem~\ref{thm:lower} says why. That
bound restricts the iteration in no way, so it binds accelerated methods too: nothing in
the model goes below $\gamma$. Above the threshold Theorem~\ref{thm:main} already sits
there, so acceleration has nothing to win and strictly loses; at the threshold the
multipliers are $1$ and it is exactly neutral. Only strictly below, where the plain
exponent $1+\frac\gamma2$ exceeds the optimal $\gamma$, is there room and acceleration
takes part of it without closing the gap. \textbf{Acceleration pays if and only if
$\gamma<2$: its threshold is the HTMC threshold itself.} It buys fewer but sharper, hence
more expensive, steps, which is the wrong currency wherever Assumption~\ref{ass:dyadic}
charges for accuracy rather than for steps.

\section{Least squares with truncated features}
\label{sec:lsq}

Assumption~\ref{ass:dyadic} posits a family of approximations whose bias and cost trade off at a fixed
rate. We instantiate it on least squares with truncated features, where that rate is
available in closed form from the geometry of the design. The assumption can then be
measured rather than posited, and the separation Theorem~\ref{thm:main} predicts between the two
exponents can be checked. The instance is chosen for what it lets us measure and not for
being a regime where the method competes with dedicated least-squares solvers.

\subsection{The design and its rate}

Let $A\in\R^{m\times n}$ with $n\gg m$, let $b\in\R^m$, and let
\[
L(X)=\tfrac12\norm{AX-b}^2,
\qquad
\nabla L(X)=A^\top r(X),\qquad r(X)=AX-b .
\]

Order the columns by decreasing norm and assume the power law
$\norm{A_i}\asymp i^{-1/\varphi}$ with $0<\varphi<2$. The level-$k$ algorithm solves the
problem restricted to the first $k$ columns and reports its gradient,
\[
A_k(X)=A_{\le k}^\top\big(A_{\le k}X-b\big),
\]
where $A_{\le k}$ is $A$ with its last $n-k$ columns zeroed. It reads $k$ columns and
nothing else, carries no state between calls, and costs $\Theta(mk)$.

\medskip
\begin{proposition}\label{prop:lsq}
Fix $R>0$. On the ball $\norm{X-X^\ast}\le R$ the family $(A_k)_k$ satisfies
Assumption~\ref{ass:dyadic} with
\[
\gamma=\frac{2\varphi}{2-\varphi},
\qquad
c^\gamma\asymp m\,M^\gamma,
\qquad
M=\norm{A}\big(2R+\norm{X^\ast}\big)+\norm{AX^\ast-b} .
\]
The same holds for $A_{\le k}^\top r(X)$, the first $k$ coordinates of the true gradient,
with the smaller constant $M=\norm{A}R+\norm{AX^\ast-b}$. In particular the problem is
HTMC precisely when $\varphi>1$.
\end{proposition}

The proof is in Appendix~\ref{app:protocol}. In our experiments we use both the truncated residual approximation $A_K(x)$ and the exact residual approximation $A_{\leq k}^T r(X)$.

\medskip
The threshold now reads off the design: $\varphi=1$ is $\norm{A_i}\asymp1/i$, so the HTMC
regime is the one in which feature importance decays \emph{more slowly than harmonically},
and no finite prefix of the design carries most of the signal.

\begin{remark}[\textbf{The rate is read off the other truncation.}]\label{rem:reading}

The algorithm above is the
natural reading of Assumption~\ref{ass:dyadic} here and, carrying nothing from one call to the next, the only one available on an arbitrary design. It is the one the experiments run but measuring $\gamma$, however, requires using $A_{\le k}^\top r(X)$ (Figure~\ref{fig:forms}, right)\footnote{
Proposition~\ref{prop:lsq} covers both with the same rate, but only the second attains the bound at every level, the cross term of the first being not a power law but a window
$(k,k_{\max}]$ that closes, so that no exponent survives a fit through it
(Appendix~\ref{app:residual}).}.

\end{remark}

\subsection{Experiment}

\subsubsection{Protocol}

Columns are drawn Gaussian, normalized, then
rescaled so that $\norm{A_i}=i^{-1/\varphi}$ exactly. The target $X^\ast$ is drawn flat, so the discarded features carry real energy,
and $b=AX^\ast$, so the full problem interpolates. Every run starts at the origin, whose
loss $L_0=L(0)=\tfrac12\norm b^2$ is the unit in which target losses are reported below.

We consider
\begin{center}
$\varphi=1.2$, hence $\gamma=3$;\quad $m=2048$,\quad levels $\{1,8,64,512\}$,\quad
$n=131072$.
\end{center}
The level ladder is geometric with ratio $2^\gamma=8$: each level keeps eight times as many
features as the one below it, costs eight times as much, and halves its bias. That is the
normalization of Assumption~\ref{ass:dyadic}, and it is what lets the probabilities
$p_k\propto2^{-(1+\gamma/2)k}$ apply as written. Two requirements fix the rest:
 $ \bm k_{\max}\ll m$ to protect the
\emph{loss floor} of a level and  $\bm k_{\max}\ll n$ to protect the
\emph{bias} of a level, which is the size of the tail it discards. The complete sizing and design of the experiment is detailed in Appendix~\ref{app:protocol}

\subsubsection{Results}

The baseline is inexact gradient descent at a fixed level: the same oracle and the same
optimizer, which is the comparison Theorem~\ref{thm:main} makes. It picks its truncation
level from a ladder finer than the oracle's, which is confined to the levels its
probabilities presuppose, so the comparison does not favour the multilevel method by confining its rival. Both methods are
tuned over the parameters they have left, the compute reported for a target loss is the
smallest at which that method reached it, and every figure is a median over ten seeds, with the median in bold, and a band showing the interquartile range.

\begin{figure}
\begin{minipage}[t]{0.49\textwidth}
\vspace{0pt}
\centering
\includegraphics[width=\textwidth]{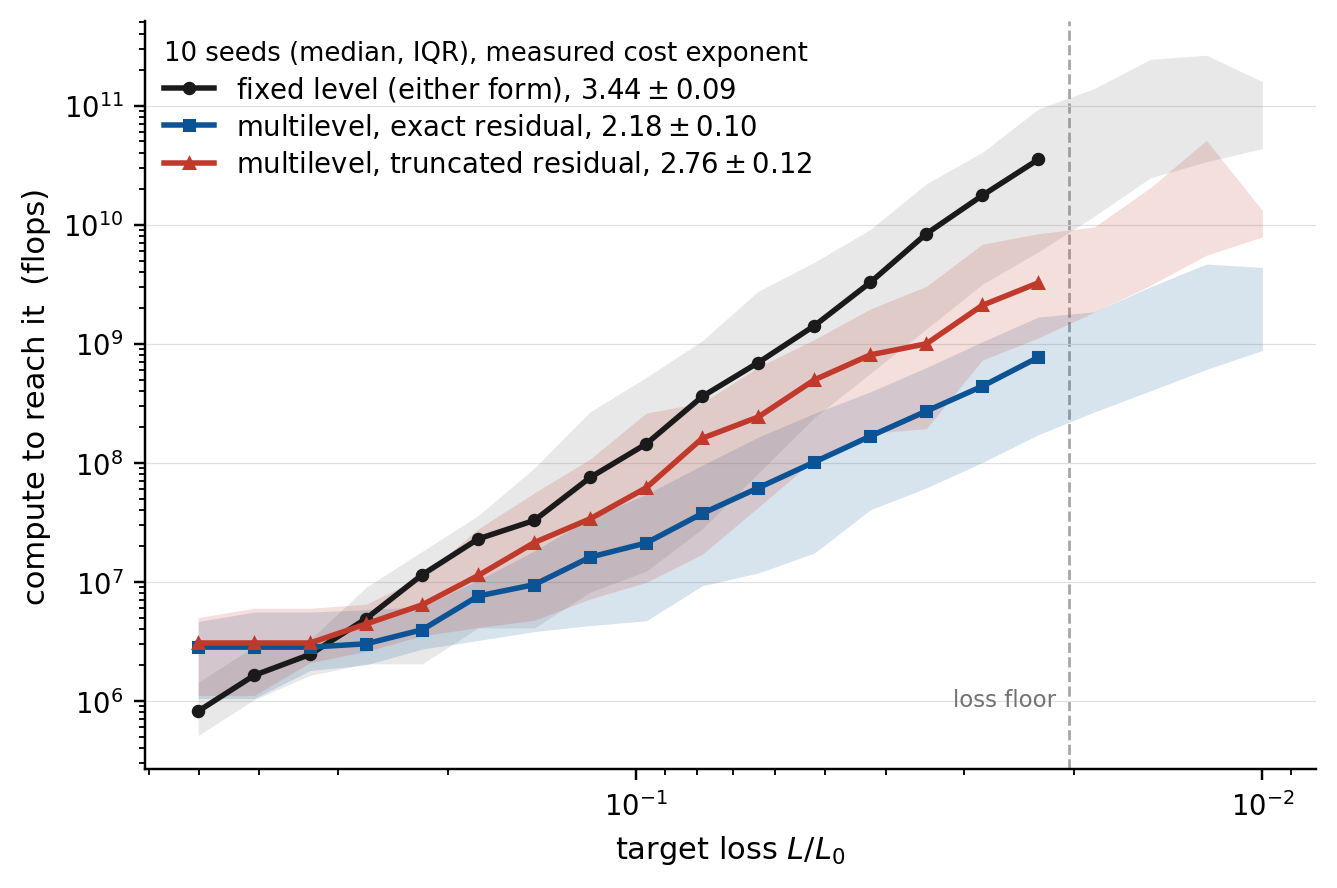}
\end{minipage}\hfill
\begin{minipage}[t]{0.49\textwidth}
\vspace{0pt}
\centering
\includegraphics[width=\textwidth]{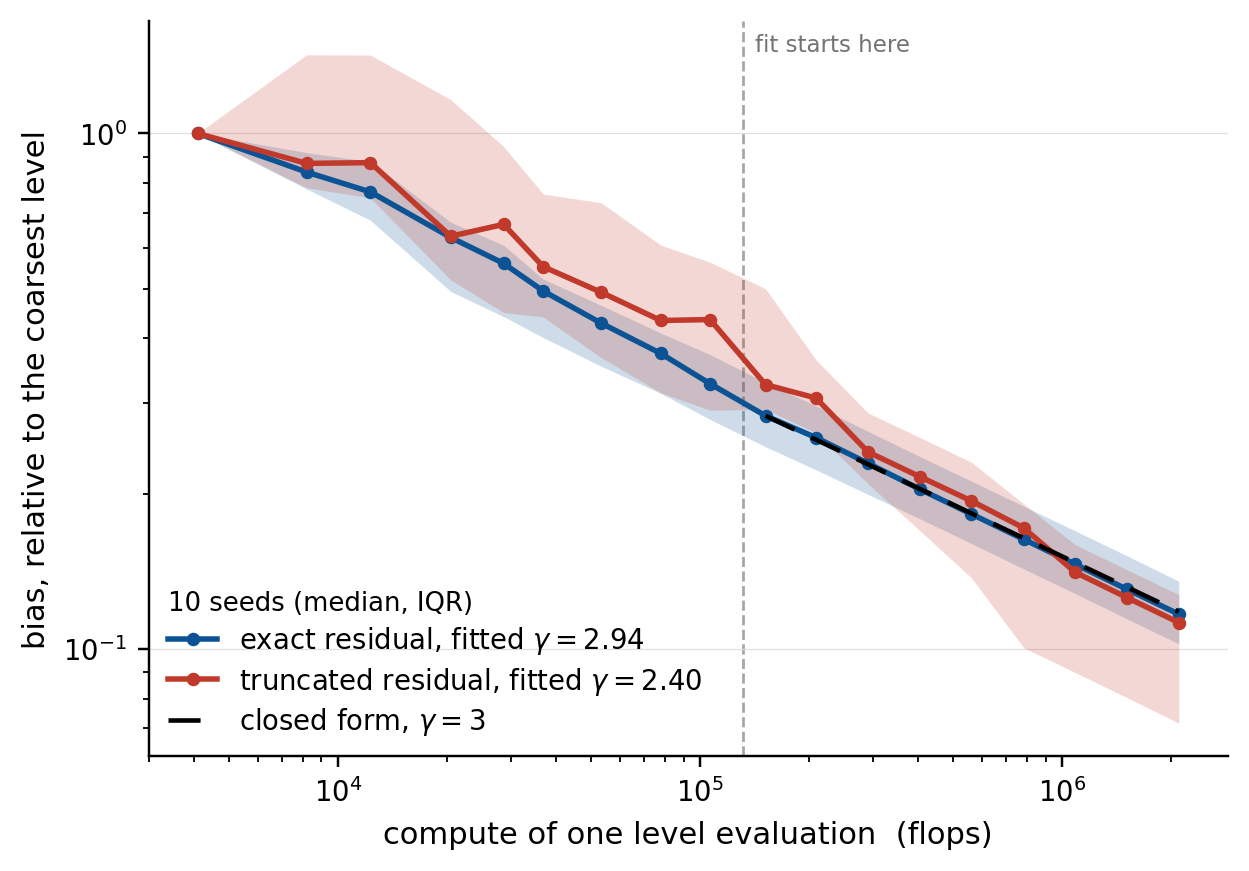}
\end{minipage}
\caption{\textbf{Left:} Compute to reach a target loss, as a fraction of $L_0$, both methods driven by
the algorithm of Proposition~\ref{prop:lsq}. The dashed line is
the loss floor of the truncated problem, which here coincides with
$\varepsilon=\mu d_0^2$ and past it the strongly convex column of Theorem~\ref{thm:main} becomes the binding one.  \textbf{Right:} bias of
each level against the compute of evaluating it, for both truncations, with each seed
normalised by its own coarsest bias, against the closed-form slope.
}
\label{fig:forms}
\end{figure}

Figure~\ref{fig:forms} is the comparison the assumption delivers: the multilevel exponent is the smaller on every
seed, $2.76\pm0.12$ against $3.44\pm0.09$, and the advantage reaches $10.9\times$ at the
tightest target both methods reach. It also shows what happens if we consider the implementation with the exact residual from Remark~\ref{rem:reading} on which $\gamma$ can be measured exactly, and with which the same run is cheaper: $2.18\pm0.10$, and $46\times$ at that same target.

On every seed the bias follows a power law with $R^2\ge0.999$ (Figure~\ref{fig:forms}, right), and
\[
\gamma_{\mathrm{meas}}=2.963\pm0.007
\qquad\text{against}\qquad
\frac{2\varphi}{2-\varphi}=3,
\]

the remaining $1.2\%$ closing as the tail of the design is lengthened
(Appendix~\ref{app:protocol}). Assumption~\ref{ass:dyadic} is measured here rather than posited, and its
exponent is the one the geometry predicts. The same fit on the truncated residual returns $2.04\pm0.26$, its per-seed values ranging
from $1.01$ to $2.99$ where the exact one stays between $2.92$ and $2.99$.

All three exponents run under the $\gamma+1$ and $\gamma$ the theorem allows them
(Appendix~\ref{app:protocol}). Neither advantage is there from the start: at loose targets
the fixed level is in fact \emph{cheaper}, as the telescope from the multilevel pays its overhead for nothing,
and the gap opens only as the target tightens.

Least squares was chosen for what it lets us check, not for competing. Truncation there
restricts to a subspace, so the loss gap of a level is the square of its gradient bias: a
solver minimizing the level-$k$ problem needs only $\varepsilon^{-\gamma/2}$, and on a
quadratic a Krylov method beats both curves above. What it buys instead is a $\gamma$ in
closed form, measurable rather than assumed.

\section{Conclusion}

In this work we showed that when the gradient of a convex objective is reachable only
through a hierarchy of approximations whose compute grows like $\delta^{-\gamma}$ in the accuracy $\delta$ and when $\gamma>2$, reaching loss $\varepsilon$ costs no more than a single gradient evaluation at the accuracy the problem demands, that is $\varepsilon^{-\gamma}$ rather than a cost of $\varepsilon^{-(\gamma+1)}$ for fixed-accuracy gradient descent. We accomplish this using a randomized multilevel oracle turning the deterministic biased hierarchy into a weakly biased estimator whose bias and variance are priced independently
and running plain inexact gradient descent. The resulting bound does not depend on the step size, and hence not on the smoothness constant, which makes it a functional of the underlying gradient flow rather than of any discretization of it. A lower bound in the same oracle model shows that the rate of $\varepsilon^{-\gamma}$ is optimal, and since the argument constrains the iteration in no way, it also binds accelerated methods. We used a simple least-squares instance whose rate is available in closed form to measure the cost assumption rather than posit it and implement the described algorithm. 

A remaining question is whether the advantages of this method survive in the non-convex setting. There is a great potential in using this method to speed up the training of DNNs, whose gradient evaluations are very costly. Given the existing evidence \citep{jacot2025norms} that DNNs themselves lie in the HTMC regime ($\gamma>2$), it seems likely that their gradient does as well.

\bibliography{iclr2027_conference}
\bibliographystyle{iclr2027_conference}

\newpage

\subsection*{AI use statement}

We used generative AI tools for writing assistance and for the implementation of the
numerical experiment. Prose was drafted by the authors and revised with AI assistance for
concision and clarity, and the \LaTeX{} preparation of tables and figures was likewise
AI-assisted. The code producing the experiment of Section~\ref{sec:lsq} was written with AI
assistance, function by function, and reviewed by the authors as it was written. We also
used AI tools to help locate related work, and checked every reference against its source.

We did not use generative AI for the research questions, the statements, or the proofs of
this paper. All AI-assisted work has been reviewed, and the correctness of the code does
not rest on that review alone: the unbiasedness of the estimator, the flop accounting, and
the scaling law the instance is meant to exhibit are each checked numerically against an
independent prediction, and those checks are reported in Appendix~\ref{app:protocol}. We
take responsibility for the final content of this work, including text, claims and
artifacts produced with the aid of generative AI.

\subsection*{Ethics statement}
This work is theoretical, and its experiment uses synthetic data generated by the procedure
described in Section~\ref{sec:lsq}. It involves no human subjects, no personal or sensitive
data, and releases no dataset. We are aware of no ethical concern specific to it beyond
those attaching to optimization methods in general.

\appendix

\section{Related work}
\label{app:related}

\paragraph{Multilevel estimators.}
Proposition~\ref{prop:oracle} is the randomized telescope of the multilevel Monte Carlo
literature \citep{heinrich2001,giles2008,giles2015}, in the randomized forms of
\citet{mcleish2011}, \citet{rheeglynn2015} and \citet{blanchetglynn2015}. What that
literature asks is when an \emph{unbiased} estimator with finite variance and finite
expected cost exists, and the answer is that it does when the increments decay faster
than the cost of a level grows. The split \eqref{eq:split} is that comparison, and the
condition is $\gamma<2$. The HTMC regime is its complement, and there no such estimator
exists: whichever level probabilities one chooses, either the variance or the expected
cost diverges. We therefore truncate the telescope deliberately. The bias this leaves is
not a defect to be removed but the quantity the optimization has to budget for, and
Theorem~\ref{thm:main} says that the budget is affordable. The same estimator drives the
multilevel Euler--Maruyama scheme of \citet{jacot2026mlem}, where simulating a diffusion
to accuracy $\varepsilon$ costs what one evaluation of the drift at that accuracy costs.
Our result is the optimization counterpart of that one.

\paragraph{Inexact and biased first-order methods.}
Optimization under an oracle that returns an approximate gradient is classical
\citep{daspremont2008,devolder2014,schmidt2011}, and so is stochastic optimization under
a biased oracle \citep{ajalloeianstich2020,hu2020bsgd,demidovich2023}. In all of it the
inexactness is given and the question is how it propagates through the method. Here it is
bought, and the question is what it costs. The two questions meet only through a cost
model, which is what Assumption~\ref{ass:dyadic} supplies. One consequence of the older
line is used directly in Section~\ref{sec:nesterov}: under an inexact oracle the bias of
an accelerated method accumulates over the horizon instead of staying bounded
\citep{devolder2014}, which is the middle term of \eqref{eq:acc-cvx}. What
Proposition~\ref{prop:acc-cvx} adds is its price.

\paragraph{Biased oracles with a cost model.}
The framework of \citet{huwangchenhe2024} is the closest to ours. 
The main difference is that this paper assumes random oracles in the first place, whereas we assume deterministic oracles and show how to upgrade them into random oracles. More precisely, \citet{huwangchenhe2024} assume a very specific structure of random oracles at different levels (or approximating the difference between two levels) inspired by practical examples. They then propose four multilevel estimators that combine these oracles, and compare these different methods. We do not make the same complex assumptions on how the different levels correlate (our levels are deterministic), but we do end up with relatively similar multilevel structures.

A second important difference is that the resulting estimators in \citet{huwangchenhe2024} always have a variance that can be halved by doubling the computational cost (or a finite variance, in which case this same variance/compute tradeoff arises implicitly when choosing the learning rate: halving the learning rate is equivalent to halving the variance at the cost of doubling the compute). In other terms, their estimators are either weakly biased (and therefore $\gamma>2$) or unbiased with $\gamma=2$. The ETMC regime ($\gamma<2$), where the compute for an unbiased estimator with variance $\sigma^2$ has cost $c^\gamma\sigma^{-\gamma}$, is not visible in their setting. This difference is reflected in the fact that in the strictly convex case, they never guarantee a convergence rate that is faster than $\epsilon^{-1}$, whereas we can guarantee faster rates in the ETMC regime. Similarly no rate faster than $2$ appears in the convex case, whereas we observe rates between $\frac{1}{2}$ and $2$. This also explains why they did not use any accelerated method as it would be useless in their setting.

A final, but less fundamental difference, is that they start from a bound on how much the loss values differ from their approximation, whereas we start from a bound on how much their gradient differ. There is no clean equivalence between these two assumptions, instead we have the two implications
\begin{enumerate}
    \item If $\left\Vert C_{\epsilon}-C\right\Vert _{\infty}\leq\epsilon$ and
$C,C_{k}$ are $\beta$-smooth, then $\left\Vert \nabla C_{\epsilon}-\nabla C\right\Vert _{\infty}\leq\sqrt{8\beta\epsilon}$.

    \item If $\left\Vert \nabla C_{\epsilon}-\nabla C\right\Vert _{\infty}\leq\delta$
over a domain of diameter $D$, then $\left\Vert (C_{\epsilon}+K)-C\right\Vert _{\infty}\lesssim D\delta$.
\end{enumerate}
This suggest that if $\epsilon$ and $\delta$ are the errors for the cost and gradient respectively, then the two match up to a power of two: $\epsilon \lesssim \delta \lesssim \sqrt{\epsilon}$. This further limits the comparability of their paper to ours.

\paragraph{Buying accuracy, and accelerating.}
\citet{vandesselglineur2023} study first-order methods whose oracle can be queried at any
accuracy $\delta$ for a cost growing as $\delta^{-r}$, and give the optimal sequence of
accuracies in closed form. Their oracle is deterministic and exposes one knob.
Proposition~\ref{prop:schedule} schedules a pair, and the gain is of another nature: not
a better constant but the disappearance of the logarithm. Random inexactness of this kind
is named there as an open direction. \citet{dvurechensky2016} analyse acceleration under
an oracle that is both biased and noisy, and the split \eqref{eq:acc-cvx} we take as
given is their Theorem~3.4. What Section~\ref{sec:nesterov} adds is again the cost model:
once each channel carries a price, balancing them against $\varepsilon$ makes the horizon
appear in the bill, and it appears with a positive exponent. Multilevel estimators have
also been used inside min-max and variational-inequality methods
\citep{alacaoglu2024}, there to control noise rather than to price accuracy.

\section{Proofs}

\subsection{Proof of Proposition~\ref{prop:oracle}}\label{app:oracle}
%=======================================================================

\subsubsection{A sufficient choice of probabilities $p_k$}
Let $B_k\sim\mathrm{Bernoulli}(p_k)$ be independent and set
\begin{equation}
\tilde f_{\delta,\sigma}=\sum_{k=k_{\min}}^{k_{\max}}\frac{B_k}{p_k}\big(A_k-A_{k-1}\big),
\qquad
p_k=\min\Big\{C\,2^{-(1+\frac\gamma2)k},\,1\Big\},
\end{equation}
for a constant $C$ to be chosen. Since $k_{\min}-1\le-\log_2c$, the algorithm
$A_{k_{\min}-1}$ has compute at most $c^\gamma2^{\gamma(k_{\min}-1)}\le1$ and we take it to
be the constant $0$ function.

\paragraph{Bias.} The sum telescopes in expectation,
\[
\E\big[\tilde f_{\delta,\sigma}\big]=\sum_{k=k_{\min}}^{k_{\max}}\big(A_k-A_{k-1}\big)
=A_{k_{\max}}-A_{k_{\min}-1}=A_{k_{\max}},
\]
so that $\norm{\E[\tilde f_{\delta,\sigma}]-f}\le2^{-k_{\max}}=\delta$, whatever the $p_k$.

\paragraph{Variance.} Since
$\norm{A_k-A_{k-1}}\le\norm{A_k-f}+\norm{f-A_{k-1}}\le2^{-k}+2^{-k+1}=3\cdot2^{-k}$ and the
summands of \eqref{eq:mlmc} are independent,
\[
\E\big\lVert\tilde f_{\delta,\sigma}-\E\tilde f_{\delta,\sigma}\big\rVert^2
=\sum_{k=k_{\min}}^{k_{\max}}\frac{1-p_k}{p_k}\norm{A_k-A_{k-1}}^2
\ \le\ 9\sum_{k=k_{\min}}^{k_{\max}}\frac{2^{-2k}}{p_k}
\ \le\ \frac9C\sum_{k=k_{\min}}^{k_{\max}}2^{(\frac\gamma2-1)k},
\]
where we used $p_k^{-1}\le C^{-1}2^{(1+\frac\gamma2)k}$, which holds in both branches of
the minimum in \eqref{eq:mlmc}.

\paragraph{Compute.} Using instead $p_k\le C2^{-(1+\frac\gamma2)k}$,
\[
\E\big[\Cost(\tilde f_{\delta,\sigma})\big]=\sum_{k=k_{\min}}^{k_{\max}}p_k\Cost(A_k)
\ \le\ c^\gamma\sum_{k=k_{\min}}^{k_{\max}}p_k2^{\gamma k}
\ \le\ C\,c^\gamma\sum_{k=k_{\min}}^{k_{\max}}2^{(\frac\gamma2-1)k}.
\]

\paragraph{The geometric sum.} Since $\gamma>2$ the ratio $2^{\frac\gamma2-1}$ exceeds $1$
and the sum is dominated by its top term,
\[
\sum_{k=k_{\min}}^{k_{\max}}2^{(\frac\gamma2-1)k}
\ \le\ \frac{2^{(\frac\gamma2-1)k_{\max}}}{1-2^{1-\frac\gamma2}}
\ =\ \frac{\delta^{1-\frac\gamma2}}{1-2^{1-\frac\gamma2}} .
\]

\paragraph{Conclusion.} Choosing
$C=\dfrac{9\,\delta^{1-\gamma/2}}{\big(1-2^{1-\gamma/2}\big)\sigma^2}$ makes the variance at
most $\sigma^2$, and then
\[
\E\big[\Cost(\tilde f_{\delta,\sigma})\big]
\ \le\ \frac{9}{\big(1-2^{1-\gamma/2}\big)^2}\,c^\gamma\,\delta^{2-\gamma}\,\sigma^{-2},
\]
which is the claim with $C_\gamma=9\big(1-2^{1-\gamma/2}\big)^{-2}$. \qed

\begin{remark}
The clamping $p_k\le1$ in \eqref{eq:mlmc} costs nothing: both bounds above use only
$p_k\le C2^{-(1+\gamma/2)k}$ and $p_k^{-1}\le C^{-1}2^{(1+\gamma/2)k}$, and each holds
whichever branch of the minimum is active. The clamp does become binding for $\gamma\le2$, where the sum is dominated by $k_{\min}$
rather than $k_{\max}$, and it is then the other two branches of
Proposition~\ref{prop:oracle} that it produces.

\end{remark}

\subsubsection{The optimal probabilities, and the other two branches}
\label{app:etmc}

The computation above fixed the shape $p_k\propto2^{-(1+\gamma/2)k}$ in advance. It is in
fact the optimal one, and seeing why is what produces the remaining two branches. Write
$V_k=\norm{A_k-A_{k-1}}_\infty^2\le9\cdot4^{-k}$ for the squared increment, bounded as
above, and $C_k\le2c^\gamma2^{\gamma k}$ for the compute of drawing it, two evaluations.
The bias is $\delta$ whatever the $p_k$, so only
\[
\Var\tilde f=\sum_k\Big(\frac1{p_k}-1\Big)V_k\le\sum_k\frac{V_k}{p_k},
\qquad
\E[\Cost]=\sum_kp_kC_k
\]
are at stake.

\paragraph{The unconstrained optimum.} Minimizing $\sum_kp_kC_k$ subject to
$\sum_kV_k/p_k\le\sigma^2$ over $p_k>0$ is solved by a Lagrange multiplier:
$p_k\propto\sqrt{V_k/C_k}$, and saturating the constraint gives
\begin{equation}\label{eq:giles}
p_k=\frac{1}{\sigma^2}\sqrt{\frac{V_k}{C_k}}\sum_j\sqrt{V_jC_j},
\qquad
\sum_kp_kC_k=\frac{1}{\sigma^2}\Big(\sum_k\sqrt{V_kC_k}\Big)^{\!2}.
\end{equation}
With the bounds above $\sqrt{V_kC_k}\asymp c^{\gamma/2}2^{(\frac\gamma2-1)k}$ and
$p_k\propto2^{-(1+\frac\gamma2)k}$, which is the shape used above. Note that $p_k$
decreases in $k$.

\paragraph{Where the clamp binds.} A probability cannot exceed one, and since $p_k$
decreases, $p_k\le1$ binds on an initial segment $k\le k_\ast$. There every $B_k$ equals
one and the increments telescope,
\[
\sum_{k\le k_\ast}(A_k-A_{k-1})=A_{k_\ast},
\]
a single deterministic evaluation at accuracy $\delta_\ast=2^{-k_\ast}$, of compute
$(c/\delta_\ast)^\gamma$ and of zero variance. Above $k_\ast$ the probabilities are
interior and \eqref{eq:giles} applies to the remaining levels alone. Optimizing over the
position of the cut gives \eqref{eq:split}, with
$\Sigma=\sum_{\delta\le2^{-k}\le\delta_\ast}2^{(\frac\gamma2-1)k}$, a geometric sum of
ratio $2^{\frac\gamma2-1}$. Everything now turns on that ratio.

\paragraph{$\gamma>2$.} The ratio exceeds one and the sum is dominated by its finest term,
$\Sigma\asymp\delta^{1-\gamma/2}$, which does not depend on $\delta_\ast$. The randomized
term of \eqref{eq:split} is therefore fixed while the prefix $(c/\delta_\ast)^\gamma$
decreases in $\delta_\ast$, so the optimum takes no prefix at all and returns the branch
proved above. The clamped prefix that the explicit $C$ of that proof produces is not this
optimal cut, but above the threshold its cost is dominated by the randomized term, which is
why the two agree.

\paragraph{$\gamma<2$.} The ratio is below one and the sum is dominated by its coarsest
term, $\Sigma\asymp\delta_\ast^{1-\gamma/2}$. Both terms of \eqref{eq:split} now move with
$\delta_\ast$ and in opposite directions, the prefix decreasing and the randomized term
increasing since $2-\gamma>0$. They balance at
\[
\Big(\frac{c}{\delta_\ast}\Big)^{\!\gamma}=\frac{c^\gamma}{\sigma^2}\delta_\ast^{2-\gamma}
\iff\delta_\ast=\sigma,
\]
and either term then equals $(c/\sigma)^\gamma$. The bias $\delta$ has left the bound:
below the threshold, accuracy is bought by computing deterministically down to $\sigma$ and
treating the remainder as noise, and the multilevel construction buys nothing.

\paragraph{$\gamma=2$.} The ratio is one and $\Sigma$ merely counts its terms,
$\Sigma=\log_2(\delta_\ast/\delta)$. Balancing as above still gives $\delta_\ast\asymp\sigma$
and $\Cost\asymp c^2\sigma^{-2}\log^2(\sigma/\delta)$. The logarithm is squared because
$\Sigma$ enters \eqref{eq:split} through its square.

\paragraph{Uniformity in $\gamma$.} The estimate $\Sigma\asymp\delta^{1-\gamma/2}$ of the
branch $\gamma>2$ hides a constant $(1-2^{1-\gamma/2})^{-1}$, of order $(\gamma-2)^{-1}$,
and is therefore not uniform as $\gamma$ approaches the threshold. Two bounds compete: that
one, and the crude $\Sigma\le(K+1)\,\delta^{1-\gamma/2}$ taking every term at the value of
the largest, with $K=\log_2(\delta_\ast/\delta)$. Neither dominates, and
\[
\Sigma\ \asymp\ \delta^{1-\gamma/2}\,
\min\big\{K,\ \big(1-2^{1-\gamma/2}\big)^{-1}\big\} ,
\]
the geometric entry binding when $(\gamma-2)K\gtrsim1$ and the counting one otherwise.
Setting $\gamma=2$ in the branch $\gamma>2$ therefore does not return the branch
$\gamma=2$: the former presumes the geometric entry, which asks
$(\gamma-2)\log_2(\sigma/\delta)\gg1$, and at the threshold the constant
$C_\gamma\asymp(\gamma-2)^{-2}$ is replaced by $\log^2(\sigma/\delta)$, finite where it is
not.

\paragraph{The exponent in $\sigma$.} At fixed $\delta$ the three branches read
$\sigma^{-2}$, $\sigma^{-2}\log^2$ and $\sigma^{-\gamma}$, so $s=\min(2,\gamma)$: variance
is bought either by randomizing, at $\sigma^{-2}$, or by computing deterministically to
accuracy $\sigma$, at $(c/\sigma)^\gamma$, and the estimator takes whichever is cheaper.
$\gamma=2$ is where the two mechanisms cost the same. \qed

\subsection{Proof of Theorem~\ref{thm:main}}\label{app:main}

Conditionally on the current iterate $y_t$, decompose the oracle into its exact, biased and
noisy parts,
\[
\tilde f_{\delta,\sigma}(y_t)=f(y_t)+b+\xi,
\qquad
b:=\E\big[\tilde f_{\delta,\sigma}(y_t)\mid y_t\big]-f(y_t),
\qquad
\xi:=\tilde f_{\delta,\sigma}(y_t)-\E\big[\tilde f_{\delta,\sigma}(y_t)\mid y_t\big],
\]
so that $b$ is deterministic given $y_t$ with $\norm b\le\delta$, while $\xi$ is centred with
$\E[\norm\xi^2\mid y_t]\le\sigma^2$. We write $a_t=\big(\E\norm{y_t-y^\ast}^2\big)^{1/2}$ and
$\tilde\sigma^2=\sigma^2+2\delta^2$.

\begin{lemma}[One step]\label{lem:onestep}
Let $y_{t+1}=y_t-\eta g$ with $g=f(y_t)+b+\xi$ as above and $\eta\le\frac1{4\beta}$. Then
\[
\E\big[\norm{y_{t+1}-y^\ast}^2\mid y_t\big]\ \le\
(1-\eta\mu)\norm{y_t-y^\ast}^2-\eta\big(\mathscr{L}(y_t)-\mathscr{L}^\ast\big)
+2\eta\delta\norm{y_t-y^\ast}+\eta^2\tilde\sigma^2 ,
\]
with $\mu=0$ in the merely convex case.
\end{lemma}

\begin{proof}
Expanding and conditioning on $y_t$, so that $\E[\ip{f(y_t)+b}{\xi}\mid y_t]=0$,
\[
\E\big[\norm{y_{t+1}-y^\ast}^2\mid y_t\big]
=\norm{y_t-y^\ast}^2-2\eta\ip{f(y_t)+b}{y_t-y^\ast}
+\eta^2\Big(\norm{f(y_t)+b}^2+\E\big[\norm\xi^2\mid y_t\big]\Big).
\]
We bound the three remaining terms:
\begin{align*}
-2\eta\ip{f(y_t)}{y_t-y^\ast}
&\le-2\eta\big(\mathscr{L}(y_t)-\mathscr{L}^\ast\big)-\eta\mu\norm{y_t-y^\ast}^2, \\
-2\eta\ip{b}{y_t-y^\ast}
&\le 2\eta\delta\norm{y_t-y^\ast}, \\
\norm{f(y_t)+b}^2\ \le\ 2\norm{f(y_t)}^2+2\norm b^2
&\le 4\beta\big(\mathscr{L}(y_t)-\mathscr{L}^\ast\big)+2\delta^2,
\end{align*}
where the first line is $\mu$-strong convexity, the second is Cauchy--Schwarz, and the
third uses $\norm{f(y_t)}^2\le2\beta(\mathscr{L}(y_t)-\mathscr{L}^\ast))$, itself obtained by evaluating
$\beta$-smoothness at $z=y-\beta^{-1}\nabla \mathscr{L}(y_t)$. Together with
$\E[\norm\xi^2\mid y_t]\le\sigma^2$, the coefficient of $\mathscr{L}(y_t)-\mathscr{L}^\ast$ is
\[
-2\eta+4\beta\eta^2=-2\eta\big(1-2\beta\eta\big)\ \le\ -\eta
\qquad\text{whenever}\qquad \eta\le\tfrac1{4\beta}. \qedhere
\]
\end{proof}

Taking total expectations gives the \emph{master recursion}
\begin{equation}\label{eq:master-app}
a_{t+1}^2\ \le\ (1-\eta\mu)\,a_t^2-\eta\,\E\big[\mathscr{L}(y_t)-\mathscr{L}^\ast\big]
+2\eta\delta\,a_t+\eta^2\tilde\sigma^2 .
\end{equation}

\subsubsection{Convex case}

Set $\mu=0$ in \eqref{eq:master-app}, rearrange and sum over $t<T$; the differences
$a_t^2-a_{t+1}^2$ telescope to at most $d_0^2$, so after dividing by $\eta T$,
\begin{equation}\label{eq:cvx-master}
\frac1T\sum_{t<T}\E\big[\mathscr{L}(y_t)-\mathscr{L}^\ast\big]
\ \le\ \frac{d_0^2}{\eta T}+2\delta\,\bar a+\eta\tilde\sigma^2,
\qquad \bar a:=\frac1T\sum_{t<T}a_t ,
\end{equation}
and Jensen's inequality transfers the bound to $\mathscr{L}(\bar y_T)$. Choose
\begin{equation}\label{eq:cvx-params}
T=\Big\lceil\frac{4d_0^2}{\eta\varepsilon}\Big\rceil,
\qquad
\delta=\frac{\varepsilon}{16\,d_0},
\qquad
\sigma^2=\frac{\varepsilon}{4\eta},
\end{equation}
which give the substitution identities
\begin{equation}\label{eq:cvx-ids}
\eta T=\frac{4d_0^2}{\varepsilon},
\qquad
\frac{d_0^2}{\eta T}=\frac\varepsilon4,
\qquad
\eta\sigma^2=\frac\varepsilon4,
\qquad
\eta\delta T=\frac{d_0}{4} .
\end{equation}

\begin{lemma}[The iterates stay in a ball]\label{lem:radius}
Under \eqref{eq:cvx-params}, $a_t\le2d_0$ for every $t\le T$.
\end{lemma}

\begin{proof}
Strong induction. The case $t=0$ is clear. Assume $a_s\le2d_0$ for all $s\le t$. Dropping
the negative loss term from \eqref{eq:master-app} and summing from $s=0$ to $t\le T-1$,
\[
a_{t+1}^2\ \le\ d_0^2+2\eta\delta T\,(2d_0)+T\eta^2\sigma^2+2T\eta^2\delta^2 .
\]
By \eqref{eq:cvx-ids} the second term is $4(\eta\delta T)d_0=d_0^2$ and the third is
$(\eta T)(\eta\sigma^2)=d_0^2$. For the last, the second branch of \eqref{eq:eta} gives
$\eta\varepsilon\le d_0^2$ and hence
\[
2T\eta^2\delta^2=2(\eta\delta T)(\eta\delta)
=2\cdot\frac{d_0}{4}\cdot\frac{\eta\varepsilon}{16d_0}
=\frac{\eta\varepsilon}{32}
\ \le\ \frac{d_0^2}{32}.
\]
Altogether $a_{t+1}^2\le3d_0^2+\frac{d_0^2}{32}<4d_0^2$, closing the induction.
\end{proof}

\paragraph{Loss.} Inserting $\bar a\le2d_0$ and \eqref{eq:cvx-ids} into
\eqref{eq:cvx-master},
\[
\E\big[\mathscr{L}(\bar y_T)-\mathscr{L}^\ast\big]
\ \le\ \underbrace{\frac{d_0^2}{\eta T}}_{=\varepsilon/4}
+\underbrace{4\delta d_0}_{=\varepsilon/4}
+\underbrace{\eta\sigma^2}_{=\varepsilon/4}
+\underbrace{2\eta\delta^2}_{\le\varepsilon/128}
\ <\ \varepsilon,
\]
the last term being controlled, again by $\eta\varepsilon\le d_0^2$, through
$2\eta\delta^2=\frac{\eta\varepsilon^2}{128\,d_0^2}\le\frac{\varepsilon}{128}$.

\paragraph{Cost.} By Proposition~\ref{prop:oracle} and $\sigma^{-2}=4\eta/\varepsilon$,
\[
\E[\Cost]=T\cdot C_\gamma c^\gamma\delta^{2-\gamma}\sigma^{-2}
=\frac{4d_0^2}{\eta\varepsilon}\cdot C_\gamma c^\gamma
\Big(\frac{\varepsilon}{16d_0}\Big)^{2-\gamma}\cdot\frac{4\eta}{\varepsilon}
=C_\gamma c^\gamma\,\frac{16d_0^2}{\varepsilon^2}
\Big(\frac{\varepsilon}{16d_0}\Big)^{2-\gamma},
\]
where the step size has cancelled between $T\propto\eta^{-1}$ and
$\sigma^{-2}\propto\eta$. Simplifying the numerical factor,
\[
\frac{16d_0^2}{\varepsilon^2}\Big(\frac{\varepsilon}{16d_0}\Big)^{2-\gamma}
=16\,d_0^2\,\varepsilon^{-2}\cdot\varepsilon^{2-\gamma}(16d_0)^{\gamma-2}
=16^{\gamma-1}d_0^\gamma\varepsilon^{-\gamma},
\]
so that $\E[\Cost]\le\Lambda_\gamma(cd_0/\varepsilon)^\gamma$ with
$\Lambda_\gamma=16^{\gamma-1}C_\gamma$.

The other two branches of Proposition~\ref{prop:oracle} give the remaining cases with no
other change, $T$, $\delta$ and $\sigma$ being independent of $\gamma$ and only the
per-call price differing. With $(c/\sigma)^\gamma$ one gets
$T\,c^\gamma(4\eta/\varepsilon)^{\gamma/2}\asymp c^\gamma d_0^2\,\eta^{\frac\gamma2-1}
\varepsilon^{-(1+\frac\gamma2)}$, whose exponent in $\eta$ is now negative, so one takes
the largest admissible step $\eta=\frac1{4\beta}$, available once
$\varepsilon\le4\beta d_0^2$ as in the deterministic case below, and $\beta^{1-\gamma/2}$
appears. With $c^2\sigma^{-2}\log^2\frac\sigma\delta$ one gets
$16\,c^2d_0^2\varepsilon^{-2}\log^2\frac\sigma\delta$ and
$\frac\sigma\delta=8d_0/\sqrt{\eta\varepsilon}$, so the step size survives only inside the
logarithm.

\paragraph{The deterministic oracle.} Running \eqref{eq:iter} with $A_{k_{\max}}$ instead
amounts to $\sigma=0$: the variance terms in \eqref{eq:master-app} vanish, so
Lemma~\ref{lem:radius} and the loss bound hold a fortiori, but $\eta$ is no longer free and
one takes the largest admissible step, $\eta=\frac1{4\beta}$ once
$\varepsilon\le4\beta d_0^2$, whence $T=\lceil16\beta d_0^2/\varepsilon\rceil$. Each step
costs the constant $(c/\delta)^\gamma=(16cd_0/\varepsilon)^\gamma$, so
\[
\Cost=T\Big(\frac{16cd_0}{\varepsilon}\Big)^\gamma
\le 16^{\gamma+1}\beta\,c^\gamma d_0^{\gamma+2}\,\varepsilon^{-(\gamma+1)} .
\]
Here no $\eta$ cancels, $\beta$ survives, and the exponent is $\gamma+1$. \qed

\subsubsection{Strongly convex case}

Weight \eqref{eq:master-app} by $w_t:=q^{-(t+1)}$ with $q:=1-\eta\mu$:
\[
q^{-(t+1)}a_{t+1}^2 \le\
q^{-t}a_t^2
-\eta\,w_t\,\E\big[\mathscr{L}(y_t)-\mathscr{L}^\ast\big]
+w_t\big(2\eta\delta a_t+\eta^2\tilde\sigma^2\big) .
\]
The telescoping is therefore exact, and summing over $t<T$ leaves $q^{-T}a_T^2\ge0$ on the
left, which we discard, against $q^{0}a_0^2=d_0^2$ on the right:
\[
\eta\sum_{t<T}w_t\,\E\big[\mathscr{L}(y_t)-\mathscr{L}^\ast\big]
\ \le\ d_0^2+\sum_{t<T}w_t\big(2\eta\delta a_t+\eta^2\tilde\sigma^2\big) .
\]
Dividing by $\eta W_T$ with $W_T=\sum_{t<T}w_t=\frac{q^{-T}-1}{\eta\mu}$, so that the
weights $w_t/W_T$ sum to one, and applying Jensen to
$\hat y_T=W_T^{-1}\sum_{t<T}w_ty_t$,
\begin{equation}\label{eq:sc-master}
\E\big[\mathscr{L}(\hat y_T)-\mathscr{L}^\ast\big]\ \le\
\underbrace{2\mu d_0^2q^{T}}_{\text{contraction}}
+\underbrace{2\delta\,\bar a_w}_{\text{bias}}
+\underbrace{\eta\tilde\sigma^2}_{\text{noise}},
\qquad
\bar a_w:=\frac1{W_T}\sum_{t<T}w_ta_t ,
\end{equation}
where we used $\frac{d_0^2}{\eta W_T}=\mu d_0^2\frac{q^T}{1-q^T}\le2\mu d_0^2q^T$, valid
once $q^T\le\frac12$. The three terms are the initial distance forgotten at the contraction
rate, the bias times the typical distance to the optimum, and the noise injected per unit
of flow time. Everything hinges on the fact that $\bar a_w$ is \emph{not} of order $d_0$.

\begin{lemma}[Noise floor]\label{lem:floor}
Set $\rho:=1-\frac{\eta\mu}{2}$. For any $V$ satisfying
\begin{equation}\label{eq:floor}
V^2\ \ge\ \frac{4\delta^2}{\mu^2}+\frac{2\eta\tilde\sigma^2}{\mu} ,
\end{equation}
one has $a_t^2\le\rho^td_0^2+V^2$, and in particular $a_t\le\rho^{t/2}d_0+V$, for every
$t\le T$.
\end{lemma}

\begin{proof}
Drop the negative loss term from \eqref{eq:master-app} and absorb the bias into the contraction
by Young's inequality $2\delta a\le\frac\mu2a^2+\frac{2\delta^2}{\mu}$:
\[
a_{t+1}^2\ \le\ \Big(1-\frac{\eta\mu}2\Big)a_t^2
+\frac{2\eta\delta^2}{\mu}+\eta^2\tilde\sigma^2
\ =\ \rho\,a_t^2+(1-\rho)\Big(\frac{4\delta^2}{\mu^2}+\frac{2\eta\tilde\sigma^2}{\mu}\Big)
\ \le\ \rho\,a_t^2+(1-\rho)V^2 ,
\]
since $1-\rho=\frac{\eta\mu}2$. The right-hand side is a convex combination of $a_t^2$ and
$V^2$, so $a_t^2\le\rho^td_0^2+(1-\rho^t)V^2$.
\end{proof}

Any admissible $V$ is a radius within which the iterates settle once the contraction has
run its course, a \emph{noise floor}. Its role is to replace $d_0$ in the bias term of
\eqref{eq:sc-master}: in the convex case $\bar a\asymp d_0$ forces
$\delta\asymp\varepsilon/d_0$, whereas here $\bar a_w\asymp V$ allows
$\delta\asymp\varepsilon/V$. The whole gain of part~(b) is this substitution, so one wants
the smallest admissible $V$.

\begin{remark}
Young's inequality is what makes $V$ explicit. Bounding $2\delta a_t$ by
$2\delta\max_sa_s$ instead would define $V$ implicitly through $\max_sa_s$, and the
resulting fixed point does not close while $d_0\gg V$. Absorbing the bias into the
contraction costs only a factor two in the rate.
\end{remark}

\paragraph{Hyperparameters.} Set
\begin{equation}\label{eq:sc-params}
V:=\sqrt{\frac\varepsilon\mu},
\qquad
\delta:=\frac{\sqrt{\mu\varepsilon}}8,
\qquad
\sigma^2:=\frac{\varepsilon}{8\eta},
\qquad
T:=\Big\lceil\frac8{\eta\mu}\log\frac{16d_0}{V}\Big\rceil ,
\end{equation}
and note that $V\le d_0$ is exactly the hypothesis $\varepsilon\le\mu d_0^2$ of
Theorem~\ref{thm:main}(b). Three verifications remain.

\smallskip
\emph{(i) $V$ is admissible in \eqref{eq:floor}.} Using
$\eta\le\frac1{4\beta}\le\frac1{4\mu}$ and $\tilde\sigma^2=\sigma^2+2\delta^2$,
\[
\frac{4\delta^2}{\mu^2}=\frac{V^2}{16},
\qquad
\frac{2\eta\sigma^2}{\mu}=\frac{V^2}4,
\qquad
\frac{4\eta\delta^2}{\mu}\le\frac{\delta^2}{\mu^2}=\frac{V^2}{64},
\]
and $\frac1{16}+\frac14+\frac1{64}<1$.

\smallskip
\emph{(ii) The weighted average sits at the floor, $\bar a_w<2V$.} Split the sum at
$T_0=\lceil T/2\rceil$. The early iterates carry almost no weight,
\[
\frac1{W_T}\sum_{t<T_0}w_t=\frac{q^{-T_0}-1}{q^{-T}-1}\ \le\ 2q^{T-T_0}\ \le\ 4q^{T/2},
\]
while $a_t\le d_0+V\le2d_0$ throughout; the late ones are already at the floor,
$a_t\le\rho^{T_0/2}d_0+V\le\rho^{T/4}d_0+V$ for $t\ge T_0$. Hence
\[
\bar a_w\ \le\ 8q^{T/2}d_0+\rho^{T/4}d_0+V .
\]
Both $q^{T/2}\le e^{-\eta\mu T/2}$ and $\rho^{T/4}\le e^{-\eta\mu T/8}$ are at most
$e^{-\eta\mu T/8}\le\frac{V}{16d_0}$ by the choice of $T$, so
$\bar a_w\le\frac V2+\frac V{16}+V<2V$. The same estimate gives
$q^T\le(\frac{V}{16d_0})^8\le\frac12$, as \eqref{eq:sc-master} requires.

\smallskip
\emph{(iii) The three terms of \eqref{eq:sc-master} sum to less than $\varepsilon$.}
\[
\underbrace{2\mu d_0^2q^T\le\frac{\mu V^2}8=\frac\varepsilon8}_{\text{contraction}},
\qquad
\underbrace{2\delta\bar a_w\le4\delta V=\frac\varepsilon2}_{\text{bias}},
\qquad
\underbrace{\eta\sigma^2+2\eta\delta^2\le\frac\varepsilon8+\frac{\delta^2}{2\mu}
=\frac\varepsilon8+\frac\varepsilon{128}}_{\text{noise}} .
\]

\paragraph{Cost.} From \eqref{eq:sc-params}, $\sigma^{-2}=8\eta/\varepsilon$ and
$\log\frac{16d_0}{V}=\frac12\log\frac{256\,\mu d_0^2}{\varepsilon}$, so that
\[
T\sigma^{-2}\ \le\ \Big(\frac4{\eta\mu}\log\frac{256\,\mu d_0^2}{\varepsilon}+1\Big)
\frac{8\eta}{\varepsilon}
\ \lesssim\ \frac1{\mu\varepsilon}\Big(1+\log\frac{\mu d_0^2}{\varepsilon}\Big),
\]
the step size cancelling once more, while
$\delta^{2-\gamma}=8^{\gamma-2}(\mu\varepsilon)^{-\frac{\gamma-2}2}$. \\
Multiplying the two,
\[
\E[\Cost]=T\cdot C_\gamma c^\gamma\delta^{2-\gamma}\sigma^{-2}
\ \lesssim\ 8^{\gamma-2}C_\gamma\,c^\gamma\,(\mu\varepsilon)^{-\gamma/2}
\Big(1+\log\frac{\mu d_0^2}{\varepsilon}\Big),
\]
which is the claim with $\Lambda'_\gamma=256\log2\cdot8^{\gamma-2}C_\gamma$.
. 
Below the threshold the same substitution applies: $T\sigma^{-2}$ becomes
$T\sigma^{-\gamma}$, which carries $\eta^{\frac\gamma2-1}$ since
$\sigma^2=\varepsilon/(8\eta)$, so $\eta=\frac1{4\beta}$ is taken and the conditioning
returns as $\kappa^{1-\gamma/2}$. At $\gamma=2$ it returns instead as $\log^2\kappa$, the
ratio $\sigma/\delta$ being $\sqrt{8/(\eta\mu)}$ there.
\qed

\begin{remark}
The logarithm produced here is exactly the factor removed by
Proposition~\ref{prop:schedule}: it enters as the number of steps multiplying a constant
per-step cost, through $T\sigma^{-2}$, and disappears as soon as $\delta$ and $\sigma$ are
allowed to vary along the trajectory.
\end{remark}

\begin{remark}
    For the strongly convex case, everything hinges on $\bar a_w$. A bias $\delta$ perturbs the loss by $\delta$ times
the typical distance to the optimum, and in the convex case that distance is $d_0$,
which forces $\delta\asymp\varepsilon/d_0$. Here it is not: \eqref{eq:master-app} shows
that the iterates contract geometrically until they settle within a radius
$V\asymp\delta/\mu+\sqrt{\eta\tilde\sigma^2/\mu}$ of the optimum, a \emph{noise
floor} at which the restoring force $\mu V$ ceases to dominate the bias $\delta$,
and the geometric weights place their mass on precisely that phase. Hence $\bar a_w\asymp V$ rather than $d_0$, and taking $V=\sqrt{\varepsilon/\mu}$ relaxes
the bias budget to $\delta\asymp\sqrt{\mu\varepsilon}$. Since the per-call compute
varies as $\delta^{2-\gamma}$ with a negative exponent, a coarser oracle is a
cheaper one, and this substitution is the whole of the exponent $\gamma/2$. The
hypothesis $\varepsilon\le\mu d_0^2$ says exactly $V\le d_0$. \qed

\end{remark}

\medskip

\begin{remark}[\textbf{Every accuracy is admissible}]\label{rem:acc}
    The second branch of \eqref{eq:eta} has no
effect on the bounds, which are $\eta$-free~. It only keeps $\eta T=4d_0^2/\varepsilon$
exact when $\varepsilon$ is large, so that the theorem holds for every
$\varepsilon>0$ rather than for $\varepsilon$ small enough. That range is of no
practical interest, since $\varepsilon\ge\beta d_0^2/2$ already makes $y_0$ itself an
$\varepsilon$-minimizer, but it costs nothing to cover. Part~(b) is restricted to
$\varepsilon\le\mu d_0^2$ for a different reason: that inequality reads $V\le d_0$
for the noise floor $V=\sqrt{\varepsilon/\mu}$ of the proof, and it is also exactly
the range in which the bound of~(b) beats that of~(a). The floor lies below the
starting point precisely when it is worth exploiting.
\end{remark}

\medskip
\begin{remark}[\textbf{What randomization buys, and what it does not.}]\label{rem:gain}
    Two effects should not be
confused. The noise floor relaxes the bias budget from $\varepsilon/d_0$ to
$\sqrt{\mu\varepsilon}$ and, since bias costs $\delta^{2-\gamma}$ with a negative
exponent, halves the exponent from $\gamma$ to $\gamma/2$~. This is a property of the
trajectory and benefits the deterministic method just as much. Randomization buys
something else. In the convex case it buys a full power of $\varepsilon$, from
$\gamma+1$ down to $\gamma$. In the strongly convex case it buys the conditioning:
the same method driven by $A_{k_{\max}}$ costs
$\kappa\log\frac{\mu d_0^2}{\varepsilon}\big(c/\sqrt{\mu\varepsilon}\big)^{\gamma}$,
a factor $\kappa$ more, so that the bound of part~(b) carries no condition number at
all.
\end{remark}

\subsection{Proof of Proposition~\ref{prop:schedule}}\label{app:schedule}

\paragraph{The small-step limit.} Per unit of flow time the iteration \eqref{eq:iter} makes
$1/\eta$ calls, each contributing $\eta b$ to the drift and $\eta^2\sigma^2$ to the
variance. These accumulate to a drift $b$ and a variance $\eta\sigma^2$ per unit time.
Writing $\nu^2(t):=\eta\sigma^2(t)$ for this \emph{diffusion rate}, which stays finite as
$\eta\to0$ while $\sigma^2\to\infty$, the limit of \eqref{eq:iter} is
\begin{equation}\label{eq:sde}
dy_t=-\big(f(y_t)+b_t\big)\,dt+dM_t,
\qquad
\norm{b_t}\le\delta(t),
\qquad
d\langle M\rangle_t=\nu^2(t)\,dt ,
\end{equation}
for a continuous martingale $M$. The compute rate is likewise
\[
\frac1\eta\cdot C_\gamma c^\gamma\delta(t)^{2-\gamma}\sigma(t)^{-2}
=C_\gamma c^\gamma\,\delta(t)^{2-\gamma}\nu(t)^{-2},
\]
again free of $\eta$ and with the schedule of Proposition~\ref{prop:schedule},
$\nu^2(t)=\frac\varepsilon4e^{\mu(T_c-t)/4}$.

\paragraph{Step 1: the master inequality.} Write $a_t^2=\E\norm{y_t-y^\ast}^2$. Itô's
formula applied to \eqref{eq:sde} kills the martingale term and contributes the quadratic
variation, so that
\[
\frac{d}{dt}a_t^2
=-2\,\E\ip{y_t-y^\ast}{f(y_t)}-2\,\E\ip{y_t-y^\ast}{b_t}+\nu^2(t)
\ \le\ -2\,\E\big[\mathscr{L}(y_t)-\mathscr{L}^\ast\big]-\mu a_t^2+2\delta(t)a_t+\nu^2(t),
\]
using $\mu$-strong convexity for the first term and Cauchy--Schwarz together with
$\E\norm{y_t-y^\ast}\le a_t$ for the second. This is the continuous-time form of
\eqref{eq:master-app}. Absorbing the bias into the contraction by the same Young inequality
$2\delta a\le\frac\mu2a^2+\frac{2\delta^2}\mu$ used in Lemma~\ref{lem:floor},
\begin{equation}\label{eq:ct-master}
\frac{d}{dt}a_t^2\ \le\ -2\,\E\big[\mathscr{L}(y_t)-\mathscr{L}^\ast\big]-\frac\mu2a_t^2+g(t),
\qquad
g(t):=\frac{2\delta(t)^2}{\mu}+\nu^2(t) .
\end{equation}
Only two quantities remain: the horizon and the injection rate $g$.

\paragraph{Step 2: weighting and integrating.} The contraction rate left by
\eqref{eq:ct-master} is $\frac\mu2$, so weight by $e^{\mu t/2}$, which makes
$\frac{d}{dt}\big(e^{\mu t/2}a_t^2\big)\le e^{\mu t/2}\big(-2\E[\mathscr{L}(y_t)-\mathscr{L}^\ast]+g(t)\big)$.
Integrating over $[0,T_c]$, discarding $e^{\mu T_c/2}a_{T_c}^2\ge0$ and dividing by
$W:=\int_0^{T_c}e^{\mu t/2}dt=\frac2\mu\big(e^{\mu T_c/2}-1\big)$, Jensen's inequality
applied to $\hat y=W^{-1}\int_0^{T_c}e^{\mu t/2}y_t\,dt$ gives
\begin{equation}\label{eq:ct-two-terms}
2\,\E\big[\mathscr{L}(\hat y)-\mathscr{L}^\ast\big]\ \le\
\underbrace{\frac{d_0^2}{W}}_{\text{horizon}}
+\underbrace{\frac1W\int_0^{T_c}e^{\mu t/2}g(t)\,dt}_{\text{injection}} .
\end{equation}

\paragraph{Step 3: the two terms.} With the schedule of Proposition~\ref{prop:schedule},
$g(t)=\frac\varepsilon2e^{\mu(T_c-t)/4}$, so that $e^{\mu t/2}g(t)$ still grows, at the
reduced rate $\frac\mu4$. Using $W\ge\frac1\mu e^{\mu T_c/2}$ once $e^{\mu T_c/2}\ge2$,
\[
\frac{d_0^2}{W}\ \le\ \mu d_0^2e^{-\mu T_c/2}\ \le\ \varepsilon,
\qquad
\frac1W\int_0^{T_c}e^{\mu t/2}g(t)\,dt
\ \le\ \frac\varepsilon2\cdot\frac{\mu e^{-\mu T_c/2}\cdot\frac4\mu e^{\mu T_c/2}}{1}
\ =\ 2\cdot\frac\varepsilon2\ =\ \varepsilon,
\]
the first by the assumption on $T_c$ and the second because
$\int_0^{T_c}e^{\mu t/2}e^{\mu(T_c-t)/4}dt=e^{\mu T_c/4}\int_0^{T_c}e^{\mu t/4}dt
\le\frac4\mu e^{\mu T_c/2}$. By \eqref{eq:ct-two-terms},
$\E[\mathscr{L}(\hat y)-\mathscr{L}^\ast]\le\varepsilon$.

\paragraph{Step 4: the compute.} Substituting the schedule,
\[
\delta(t)^{2-\gamma}\nu(t)^{-2}
=\Big(\frac{\mu\varepsilon}8\Big)^{\frac{2-\gamma}2}e^{\frac{\mu(T_c-t)(2-\gamma)}8}
\cdot\frac4\varepsilon\,e^{-\frac{\mu(T_c-t)}4}
=\Big(\frac{\mu\varepsilon}8\Big)^{\frac{2-\gamma}2}\frac4\varepsilon\,
e^{-\frac{\mu\gamma(T_c-t)}8},
\]
the exponent being negative precisely because $\frac{2-\gamma}8-\frac28=-\frac\gamma8$.
Since $\int_0^{T_c}e^{-\mu\gamma(T_c-t)/8}dt\le\frac8{\mu\gamma}$,
\[
\E[\Cost]=\int_0^{T_c}C_\gamma c^\gamma\delta(t)^{2-\gamma}\nu(t)^{-2}\,dt
\ \le\ \frac{32}{\gamma}\,8^{\frac{\gamma-2}2}\,C_\gamma\,c^\gamma\,
(\mu\varepsilon)^{\frac{2-\gamma}2-1}
\ =\ \Lambda''_\gamma\Big(\frac{c}{\sqrt{\mu\varepsilon}}\Big)^{\gamma},
\]
independently of $T_c$, with $\Lambda''_\gamma=\frac{32}\gamma8^{\frac{\gamma-2}2}C_\gamma$.
\qed

\begin{remark}[The schedule is not unique]
Write $g(t)=g_\infty e^{\theta\mu(T_c-t)/2}$ for a decay parameter $\theta\ge0$; the
proposition takes $\theta=\frac12$. Step~3 then yields
$\frac1W\int e^{\mu t/2}g\le\frac{g_\infty}{1-\theta}$, which is finite exactly for
$\theta<1$, while the exponent of Step~4 is $\frac{\theta(2-\gamma)}2-\theta=-\frac{\theta\gamma}2$,
negative for every $\theta>0$. Any $\theta\in(0,1)$ therefore works, and the two endpoints
are the two ways of reintroducing a logarithm: at $\theta=0$ the accuracy is constant and
the compute integral becomes $T_c$, which is the logarithm of
Theorem~\ref{thm:main}(b); at $\theta=1$ the injection integral becomes $T_c$ and the error
bound degrades instead.
\end{remark}

\section{The exponent is optimal}
\label{app:lower}

Appendix~\ref{app:lower}.1 and~\ref{app:lower}.2 prove Theorem~\ref{thm:lower} by
exhibiting a pair of instances on which every algorithm pays.
Appendix~\ref{app:lower}.3 proves the statement of Remark~\ref{rem:legendre}, which
exhibits nothing: under strong convexity the exponent is forced by the upper bound
itself, through the stability of the class under Legendre conjugation.

\subsection{The model}

\begin{definition}[Cost-weighted hierarchy oracle]\label{def:model}
An \emph{instance} is a pair $(\mathscr{L},\mathcal{A})$ with $\mathscr{L}$ convex
and $\beta$-smooth and $\mathcal{A}=(A_k)_k$ a family satisfying
Assumption~\ref{ass:dyadic} for $\mathscr{L}$, with the \emph{rate} of the instance being
the $\gamma$ of that assumption. An algorithm interacts with an instance in rounds.
At round $t$ it selects, as a measurable function of the answers received so far
and of its own internal randomness, a point $y_t$ and a level $k_t$. It receives
$A_{k_t}(y_t)$ and pays $c^\gamma2^{\gamma k_t}$. After a possibly random number of
rounds it outputs $\hat y$, and its \emph{cost} is $\sum_tc^\gamma2^{\gamma k_t}$.
\end{definition}

Computation between queries is free, since only the compute spent inside the oracle
is counted on either side. The estimator of Proposition~\ref{prop:oracle} is an
algorithm in this model: drawing the increment $A_k-A_{k-1}$ is two queries, at
levels $k$ and $k-1$.

\subsection{Two instances that no coarse query separates}

Let $(\mathscr{L}_0,\mathcal{A}^0)$ be any instance of rate $\gamma$ and prefactor $c$ on
$\R^d$, with minimiser $z^\ast$, and fix $\lambda>0$ and $\delta>0$. Work on $\R\times\R^d$, one dimension more than the base, writing $y=(v,z)$, and set

\begin{equation}\label{eq:lb-family}
\mathscr{L}_\pm(v,z)=\mathscr{L}_0(z)+\tfrac\lambda2v^2\pm\delta v,
\qquad
\nabla\mathscr{L}_\pm(v,z)=\big(\lambda v\pm\delta,\ \nabla\mathscr{L}_0(z)\big),
\end{equation}
both started at $y_0=(0,z^\ast)$. Each is convex and $\max(\beta,\lambda)$-smooth\footnote{Recall that $\mathscr{L}_0$ is taken $\beta$-smooth according to Theorem~\ref{thm:main}}, and
$\mu$-strongly convex whenever $\mathscr{L}_0$ is and $\lambda\ge\mu$. Their minimisers are
$y^\ast_\pm=(\mp\frac\delta\lambda,z^\ast)$, so $\norm{y_0-y^\ast_\pm}=\delta/\lambda$ for
both, and their gradients differ by $(2\delta,0)$ at every point. Flattening one direction
is what converts a gradient perturbation into a loss: the same $\delta$ on a direction
already curved at $\beta$ would buy $\delta^2/\beta$ instead of $\delta^2/\lambda$.

\begin{lemma}[No common approximate minimizer]\label{lem:sep}
For every $(v,z)$,
\[
\big(\mathscr{L}_+-\mathscr{L}_+^\ast\big)+\big(\mathscr{L}_--\mathscr{L}_-^\ast\big)
=2\big(\mathscr{L}_0(z)-\mathscr{L}_0^\ast\big)+\lambda v^2+\frac{\delta^2}\lambda
\ \ge\ \frac{\delta^2}\lambda .
\]
\end{lemma}

\begin{proof}
$\mathscr{L}_\pm(v,z)-\mathscr{L}_\pm^\ast
=(\mathscr{L}_0(z)-\mathscr{L}_0^\ast)+\frac\lambda2(v\pm\frac\delta\lambda)^2$, and
$(v+\frac\delta\lambda)^2+(v-\frac\delta\lambda)^2=2v^2+2\frac{\delta^2}{\lambda^2}$. The
base contributes a nonnegative term and is discarded.
\end{proof}

\begin{lemma}[Every level coarser than $\delta$ is blind]\label{lem:blind}
Let $k_\ast$ be the least integer with $2^{-k_\ast}<2\delta$. There are hierarchies
$\mathcal{A}^\pm$ for $\mathscr{L}_\pm$, both satisfying Assumption~\ref{ass:dyadic} with
rate $\gamma$ and prefactor $2c$, such that $A_k^+=A_k^-$ for every $k<k_\ast$.
\end{lemma}

\begin{proof}
Put $A^\pm_k(v,z)=(\lambda v,\,A^0_{k+1}(z))$ for $k<k_\ast$ and
$A^\pm_k(v,z)=(\lambda v\pm\delta,\,A^0_k(z))$ for $k\ge k_\ast$. The first is the same map
for both instances. For $k<k_\ast$ we have $2^{-k}\ge2\delta$, hence
$\delta\le2^{-k-1}$ and
\[
\norm{A^\pm_k-\nabla\mathscr{L}_\pm}^2=\delta^2+\norm{A^0_{k+1}-\nabla\mathscr{L}_0}^2
\le2^{-2k-2}+2^{-2k-2}\le2^{-2k},
\]
at cost $\Cost(A^0_{k+1})\le c^\gamma2^{\gamma(k+1)}=(2c)^\gamma2^{\gamma k}$, the extra
scalar operation being absorbed into the prefactor. For $k\ge k_\ast$ the error is that of
$A^0_k$, at most $2^{-k}$, at cost at most $(2c)^\gamma2^{\gamma k}$.
\end{proof}

The two instances are therefore not merely close below level $k_\ast$: they return the
identical vector at every point and every such level, so no algorithm separates them there,
whatever it does with the answers. Only the adjoined coordinate is left to optimise: the base is started at its own minimiser
and plays no part in the separation. Its role is not to be hard but to make the price
legitimate. On a bare quadratic the gradient is computable exactly at bounded cost, so no
hierarchy of rate $\gamma$ exists for it and Definition~\ref{def:model} would be charging
for something free. The model sells levels rather than components, here as in
Theorem~\ref{thm:main}, whose oracle also buys each $A_k$ whole.
 In particular
$\gamma_{\min}(\nabla\mathscr{L}_\pm)=\gamma_{\min}(\nabla\mathscr{L}_0)$, since the two
gradients determine one another at equal accuracy, so whenever the base is intrinsically of
rate $\gamma$ the hard pair is too, and Remark~\ref{rem:gammamin} does not apply to it.

\begin{proof}[Proof of Theorem~\ref{thm:lower}]
Take \eqref{eq:lb-family} with $\lambda=8\varepsilon/d_0^2$ and $\delta=\lambda d_0$ in the
convex case, and with $\lambda=\mu$ and $\delta=\sqrt{8\mu\varepsilon}$ in the strongly
convex case. Either way $\norm{y_0-y^\ast}=\delta/\lambda\le d_0$, the smoothness constant
is $\max(\beta,\lambda)=\beta$ under the stated bound on $\varepsilon$, and
\begin{equation}\label{eq:lb-budget}
\delta^2/\lambda=8\varepsilon .
\end{equation}

Run the algorithm on both instances with the same internal randomness and the hierarchies
of Lemma~\ref{lem:blind}. Let $\tau=\inf\{t:k_t\ge k_\ast\}$ and $E=\{\tau=\infty\}$.

\emph{The two runs agree up to and including round $\tau$.} By induction the round-$0$
choice depends only on the internal randomness and if the runs agree through round $t-1$
with every level queried there below $k_\ast$, then by Lemma~\ref{lem:blind} every answer
received was the same in both, so the round-$t$ choice, a function of those answers and of
the shared randomness, is the same as well. On $E$ the two runs therefore coincide for all
time and return the same $\hat y$, while on $E^c$ they select the same round $\tau$ and the
same level $k_\tau\ge k_\ast$, so \emph{both} pay for the fine query.

\emph{Bounding $\Pr(E)$.} The algorithm succeeds on both instances, so
$\E[\mathscr{L}_\pm(\hat y_\pm)-\mathscr{L}_\pm^\ast]\le\varepsilon$. The two gaps are
nonnegative, so restricting the expectation to $E$ only lowers it, and the two outputs are
equal on $E$, so Lemma~\ref{lem:sep} with \eqref{eq:lb-budget} bounds their sum below by
$8\varepsilon$ there. Hence
\[
2\varepsilon\ \ge\
\E\big[\mathscr{L}_+(\hat y_+)-\mathscr{L}_+^\ast\big]
+\E\big[\mathscr{L}_-(\hat y_-)-\mathscr{L}_-^\ast\big]
\ \ge\ 8\varepsilon\,\Pr(E),
\]
so $\Pr(E)\le\frac14$ and $\Pr(E^c)\ge\frac34$. The algorithm is not required to succeed on
every realization, only on average, so a positive $\Pr(E)$ is no contradiction; what the
display forbids is that $E$ be likely.

\emph{Charging for the fine query.} On $E^c$ both runs pay at least
$(2c)^\gamma2^{\gamma k_\ast}$, and $2^{k_\ast}>1/(2\delta)$ by minimality, so on either
instance
\[
\E[\Cost]\ \ge\ \tfrac34\,(2c)^\gamma(2\delta)^{-\gamma}\ =\ \tfrac34\,(c/\delta)^\gamma :
\]
the coarser threshold and the doubled prefactor cancel. Substituting the two calibrations
of $\delta$ gives the two displays of Theorem~\ref{thm:lower}.
\end{proof}

\subsection{Duality forces the exponent, without any construction}

Let $\mathscr{L}^\ast(z)=\sup_y\{\ip zy-\mathscr{L}(y)\}$. If $\mathscr{L}$ is
$\mu$-strongly convex and $\beta$-smooth then $\mathscr{L}^\ast$ is
$1/\beta$-strongly convex and $1/\mu$-smooth, and
$\mathscr{L}^{\ast\ast}=\mathscr{L}$: the class is stable and conjugation is an
involution of it. Everything rests on one identity,
\begin{equation}\label{eq:fenchel}
\nabla\mathscr{L}^\ast(z)\ =\ \arg\min_y\big\{\mathscr{L}(y)-\ip zy\big\},
\end{equation}
which says that \emph{evaluating the gradient of the conjugate is solving a tilted
copy of the original problem}. A solver for $\mathscr{L}$ is therefore an
approximation scheme for $\nabla\mathscr{L}^\ast$, and its accuracy--cost profile
becomes that scheme's scaling law.

\begin{lemma}[Tilts are free]\label{lem:tilt}
If $(\mathscr{L},\mathcal{A})$ is an instance of rate $\gamma$ and prefactor $c$,
then for every $z$ the tilted objective
$\mathscr{L}_z:=\mathscr{L}-\ip{z}{\cdot}$ is an instance with the same $\mu$,
$\beta$, $\gamma$ and $c$, through $A_k-z$.
\end{lemma}

\begin{proof}
$\nabla\mathscr{L}_z=\nabla\mathscr{L}-z$, so
$\norm{(A_k-z)-\nabla\mathscr{L}_z}=\norm{A_k-\nabla\mathscr{L}}\le2^{-k}$, and
subtracting a fixed vector costs nothing. A linear term changes neither $\mu$ nor
$\beta$.
\end{proof}

\begin{lemma}[Transfer]\label{lem:transfer}
Let $\mathscr{L}$ be $\mu$-strongly convex with a hierarchy of rate $\gamma$ and
prefactor $c$, and suppose an algorithm solves every $\mu$-strongly convex
instance to loss $\varepsilon$ at cost $O(\varepsilon^{-t})$, the implied constant
depending only on $\gamma$, $c$, $\mu$ and $\beta$. Then
$\nabla\mathscr{L}^\ast$ admits a hierarchy of rate $2t$.
\end{lemma}

\begin{proof}
For each $z$, the tilt $\mathscr{L}_z$ is by Lemma~\ref{lem:tilt} an instance with
the same $\gamma$, $c$, $\mu$ and $\beta$, so the algorithm applies to it at a cost
independent of $z$. Strong convexity gives
$\mathscr{L}_z(y)-\mathscr{L}_z^\ast\ge\frac\mu2\norm{y-y_z^\ast}^2$, so a loss
$\varepsilon$ certifies a distance $\Delta=\sqrt{2\varepsilon/\mu}$ and the cost
$O(\varepsilon^{-t})$ reads $O(\Delta^{-2t})$: this is where the factor two comes
from. Let $A'_k(z)$ be the algorithm's output on $\mathscr{L}_z$ run to distance
$2^{-k}$. By \eqref{eq:fenchel} the minimizer $y_z^\ast$ is
$\nabla\mathscr{L}^\ast(z)$, so
$\norm{A'_k(z)-\nabla\mathscr{L}^\ast(z)}\le2^{-k}$ for every $z$, at cost
$O(2^{2tk})$ which is the announced hierarchy of rate $2t$.
\end{proof}

Let
\[
\gamma_{\min}(f)\ =\ \inf\big\{\gamma>0:\ \norm{f}_{M^\gamma}<\infty\big\}
\]
be the critical exponent of the HTMC norm~\eqref{eq:htmc}, the best rate any hierarchy
for $f$ can have. It is a property of the function, where the $\gamma$ of
Assumption~\ref{ass:dyadic} is a property of a supplied family and exhibiting a
hierarchy bounds $\gamma_{\min}$ from above and never from below.

\begin{proposition}[No speedup at the intrinsic rate]\label{prop:legendre}
Let $\mathscr{L}$ be $\mu$-strongly convex and $\beta$-smooth, and let
$g=\gamma_{\min}(\nabla\mathscr{L})$. Suppose an algorithm solves every
$\mu$-strongly convex instance to loss $\varepsilon$ at cost $O(\varepsilon^{-t})$,
the implied constant depending only on the parameters of the instance. Then
$t\ge g/2$.
\end{proposition}

\begin{proof}
Lemma~\ref{lem:transfer} turns that algorithm into a hierarchy for
$\nabla\mathscr{L}^\ast$ of rate $2t$. Now $\mathscr{L}^\ast$ is
$1/\beta$-strongly convex and $1/\mu$-smooth, so Theorem~\ref{thm:main}(b) runs on
that hierarchy at loss exponent $\frac12(2t)=t$, and by Lemma~\ref{lem:tilt} it
does so on every tilt of $\mathscr{L}^\ast$ at a uniform cost. A second application
of Lemma~\ref{lem:transfer}, to $\mathscr{L}^\ast$, produces a hierarchy for
$\nabla\mathscr{L}^{\ast\ast}=\nabla\mathscr{L}$ of rate $2t$. Hence
$g=\gamma_{\min}(\nabla\mathscr{L})\le2t$.
\end{proof}

Theorem~\ref{thm:main}(b) attains the exponent $(g+\eta)/2$ for every $\eta>0$, so
$g/2$ is optimal: at the intrinsic rate, optimizing costs exactly what evaluating
the gradient once costs.

Proposition~\ref{prop:legendre} is proved under a mild relaxation of
Assumption~\ref{ass:dyadic}: Lemma~\ref{lem:transfer} builds its levels out of an
algorithm whose guarantee holds in expectation, so they are accurate with high
probability rather than surely. Restoring a sure guarantee costs a logarithmic
number of repetitions and a union bound over the horizon, which changes no
exponent. Theorem~\ref{thm:lower}, by contrast, needs no such relaxation.

\begin{remark}[The intrinsic rate is conjugation-invariant]\label{rem:invariance}
Applying Lemma~\ref{lem:transfer} with Theorem~\ref{thm:main}(b) itself turns a
hierarchy of rate $\gamma$ for $\nabla\mathscr{L}$ into one of rate
$2\cdot\frac\gamma2=\gamma$ for $\nabla\mathscr{L}^\ast$. Hence
$\gamma_{\min}(\nabla\mathscr{L}^\ast)\le\gamma_{\min}(\nabla\mathscr{L})$, and
symmetrically: the critical exponent of ~\eqref{eq:htmc} is the same for a
function and its conjugate.
\end{remark}
\begin{remark}[Two bookkeeping points in the transfer]\label{rem:transfer-cond}
Neither moves an exponent. Theorem~\ref{thm:main}(b) carries a factor
$1+\log(\mu d_0^2/\varepsilon)$ with $d_0$ the distance from the start to the
minimizer, so the constant of Lemma~\ref{lem:transfer} would depend on $z$ through
$\nabla\mathscr{L}^\ast(z)$. Computing $A'_k(z)$ from $A'_{k-1}(z)$ rather than
from scratch fixes $d_0$ at $2^{-(k-1)}$ against a target $2^{-k}$, which makes
that factor a constant, and the cost of $A'_k(z)$ is then a geometric sum over the
refinements, dominated by the last. Only the coarsest level is computed cold, from
a fixed start, at a cost carrying $\log\norm{z}$ and reading the sup-norm of
Assumption~\ref{ass:dyadic} over a bounded region absorbs it into the prefactor.
\end{remark}

\begin{remark}[Scope, and what is open]\label{rem:legendre-scope}
Proposition~\ref{prop:legendre} needs strong convexity twice over, to convert a
loss into a distance and to keep the conjugate inside the class, so only
Theorem~\ref{thm:lower} covers the merely convex case. It also lives at the
intrinsic rate, and exhibiting a hierarchy bounds $\gamma_{\min}$ from above only:
discharging its hypothesis on a concrete family is open, and
Remark~\ref{rem:gammamin} shows that on the instance of Section~\ref{sec:lsq} the
proposition is true and empty. Metric entropy settles the matter at the level of a
function class rather than of a single function, the $\varepsilon$-entropy of a
$C^\alpha$ ball being $\asymp\varepsilon^{-d/\alpha}$, at the cost of machinery we
do not develop here. Theorem~\ref{thm:lower} needs none of this, which is why it is
the statement in the body.
\end{remark}

\subsection{Tightness, and where it stops}

Nothing in the proof of Theorem~\ref{thm:lower} uses $\gamma>2$, so it holds at
every rate; but it is tight only for $\gamma\ge2$. Below the threshold the upper
bound of Corollary~\ref{cor:eta} is
$c^\gamma d_0^2\beta^{1-\gamma/2}\varepsilon^{-(1+\gamma/2)}$, and the ratio to
Theorem~\ref{thm:lower} is $(\beta d_0^2/\varepsilon)^{(2-\gamma)/2}$, equal to $1$
at $\gamma=2$ and diverging as $\gamma$ decreases. This is a statement about where
the cost sits rather than a defect of the construction. Above the threshold the
compute is dominated by the price of a single query at the accuracy the problem
demands, which is what the argument charges for. Below it the compute is dominated
by the \emph{number} of iterations, which a single-query argument cannot see. A
matching lower bound there would have to show that a constant fraction of the
iterations must be run at a fine level.

\section{Acceleration}
\label{app:acc}

We drive the standard accelerated scheme with the same oracle,
\begin{equation}\label{eq:acc-iter}
y_{t+1}=z_t-\eta\,\tilde f_{\delta,\sigma}(z_t),
\qquad
z_{t+1}=y_{t+1}+\theta_t\,(y_{t+1}-y_t),
\end{equation}
where the momentum coefficient $\theta_t$ equals $\frac{t-1}{t+2}$ in the convex case
and the constant $\frac{\sqrt\kappa-1}{\sqrt\kappa+1}$ under strong convexity.

The error of such a scheme under a jointly biased and noisy oracle splits into three
channels, and we take that split from \citet{dvurechensky2016}, whose Theorem~3.4 gives
$\Theta\big(LR^2/T^{p}+\sigma R/\sqrt T+T^{p-1}\delta\big)$ for a family indexed by
$p\in[1,2]$, accelerated at $p=2$. What has not been done is to balance those three
channels against a cost model, which is what Assumption~\ref{ass:dyadic} supplies.

Both propositions rest on one inequality, which we take from the literature, and one
lemma, which we prove because it is the bridge to Assumption~\ref{ass:dyadic}. Everything
else is substitution.

\paragraph{The inequality.} Following \citet{devolder2014}, a pair
$(f_{\delta_0},g_{\delta_0})$ is a $(\delta_0,L)$-oracle for $\mathscr{L}$ on a convex set
$Q$ when
\begin{equation}\label{eq:dgn}
0\ \le\ \mathscr{L}(y)-f_{\delta_0}(x)-\ip{g_{\delta_0}(x)}{y-x}\ \le\
\tfrac L2\norm{y-x}^2+\delta_0
\qquad\text{for all }x,y\in Q ,
\end{equation}
convexity with a slack on the left, smoothness with a slack on the right, the slack
$\delta_0$ being measured in function values. Driven by such an oracle whose calls carry
an independent zero-mean noise of variance $\sigma^2$, the accelerated method run for $T$
steps from a point at distance $R$ from the minimizer satisfies
\begin{equation}\label{eq:dg34}
\E\big[\mathscr{L}(y_T)-\mathscr{L}^\ast\big]\ \lesssim\
\frac{LR^2}{T^{2}}\ +\ \frac{\sigma R}{\sqrt T}\ +\ T\delta_0
\end{equation}
in the Euclidean setup. This is Theorem~3.4 of \citet{dvurechensky2016} at $p=2$, where
the constants are explicit. The three terms are the three channels of
Section~\ref{sec:nesterov}, the last growing with the horizon as \citet{devolder2014}
showed it must. Nothing below uses anything else about the method.

\paragraph{The bridge.} Assumption~\ref{ass:dyadic} bounds a \emph{gradient}, not a
function value, so \eqref{eq:dg34} does not apply to it as it stands.

\begin{lemma}[Conversion]\label{lem:convert}
Let $\mathscr{L}$ be convex and $\beta$-smooth and let $\bar g=\E[\tilde f_{\delta,\sigma}]$
be the mean of the oracle of Proposition~\ref{prop:oracle}, so that
$\norm{\bar g(x)-\nabla\mathscr{L}(x)}\le\delta$ for every $x$. Then
\begin{enumerate}[label=(\alph*),leftmargin=2.2em,itemsep=3pt,topsep=3pt]
\item on a convex set of diameter $D$, the pair $\big(\mathscr{L}-\delta D,\ \bar g\big)$
satisfies \eqref{eq:dgn} with $\delta_0=2\delta D$ and $L=\beta$;
\item if $\mathscr{L}$ is moreover $\mu$-strongly convex, the pair
$\big(\mathscr{L}-\delta^2/\mu,\ \bar g\big)$ satisfies \eqref{eq:dgn} on all of $\R^d$,
with $\delta_0\le2\delta^2/\mu$ and $L=2\beta$.
\end{enumerate}
\end{lemma}

\begin{proof}
Write $\Delta=\nabla\mathscr{L}(x)-\bar g(x)$, so $\norm\Delta\le\delta$, put
$a=\norm{y-x}$, and let $c$ be the constant subtracted from $\mathscr{L}$. The quantity to
bracket in \eqref{eq:dgn} is
\[
\underbrace{\mathscr{L}(y)-\mathscr{L}(x)-\ip{\nabla\mathscr{L}(x)}{y-x}}_{=:B}
\ +\ \ip{\Delta}{y-x}\ +\ c ,
\qquad |\ip{\Delta}{y-x}|\le\delta a .
\]
In case (a), $c=\delta D$ and $a\le D$: convexity gives $B\ge0$, so the quantity is at
least $\delta D-\delta a\ge0$, and smoothness gives $B\le\frac\beta2a^2$, so it is at most
$\frac\beta2a^2+2\delta D$. In case (b), $c=\delta^2/\mu$ and $a$ is free: strong
convexity gives $B\ge\frac\mu2a^2$, so the quantity is at least
$\frac\mu2a^2-\delta a+\frac{\delta^2}\mu$, a quadratic of discriminant
$\delta^2-2\delta^2<0$ and hence non-negative. Smoothness and
$\delta a\le\frac\beta2a^2+\frac{\delta^2}{2\beta}$ bound it by
$\beta a^2+\frac{\delta^2}\mu+\frac{\delta^2}{2\beta}$.
\end{proof}

Conversion (a) is linear in $\delta$ and charges a diameter, (b) is quadratic in $\delta$
and charges nothing. Since the per-call compute carries $\delta^{2-\gamma}$ with a
negative exponent, the one admitting the larger $\delta$ is the cheaper, and that is what
separates the two propositions.

\subsection{The convex case}

Run the method on the ball of radius $d_0$ about $y^\ast$, of diameter $2d_0$. By
Lemma~\ref{lem:convert}(a) the oracle is a $(4\delta d_0,\beta)$-oracle there, and
\eqref{eq:dg34} with $L=\beta$ and $R=d_0$ reads
\begin{equation}\label{eq:acc-cvx}
\E\big[\mathscr{L}(y_T)-\mathscr{L}^\ast\big]\ \lesssim\
\underbrace{\frac{\beta d_0^2}{T^2}}_{\text{accelerated}}
\ +\ \underbrace{T\,\delta d_0}_{\text{accumulated bias}}
\ +\ \underbrace{\frac{\sigma d_0}{\sqrt T}}_{\text{not accelerable}} ,
\end{equation}
where the middle term is the $O(T\delta)$ accumulation of \citet{devolder2014} and the
last is the statistical term of the optimal stochastic rate, attained by AC-SA
\citep{lan2012} and unimprovable by any first-order method \citep{nemirovski2009}.

\begin{proof}[Proof of Proposition~\ref{prop:acc-cvx}, \textbf{above the threshold}]
Setting each term of \eqref{eq:acc-cvx} to $\varepsilon$ in turn,
\[
\frac{\beta d_0^2}{T^2}=\varepsilon
\ \Rightarrow\ T\asymp d_0\sqrt{\frac\beta\varepsilon},
\qquad
T\delta d_0=\varepsilon
\ \Rightarrow\ \delta\asymp\frac{\varepsilon^{3/2}}{d_0^2\sqrt\beta},
\qquad
\frac{\sigma d_0}{\sqrt T}=\varepsilon
\ \Rightarrow\ \sigma^2\asymp\frac{\varepsilon^{3/2}\sqrt\beta}{d_0}.
\]
In the HTMC regime, the compute will be $T$ times the cost per call $C_\gamma c^\gamma\delta^{2-\gamma}\sigma^{-2}$
of Proposition~\ref{prop:oracle}. Collecting the exponent of each variable,
\[
\varepsilon:\ -\tfrac12+\tfrac32(2-\gamma)-\tfrac32=1-\tfrac{3\gamma}2,
\qquad
d_0:\ 1-2(2-\gamma)+1=2\gamma-2,
\qquad
\beta:\ \tfrac12-\tfrac{2-\gamma}2-\tfrac12=\tfrac{\gamma-2}2 ,
\]
so the compute is of order
$c^\gamma d_0^{2\gamma-2}\beta^{(\gamma-2)/2}\varepsilon^{-(3\gamma-2)/2}$. Dividing by
the bound $\Lambda_\gamma(cd_0/\varepsilon)^\gamma$ of Theorem~\ref{thm:main}(a) leaves
$d_0^{\gamma-2}\beta^{(\gamma-2)/2}\varepsilon^{1-\gamma/2}$, which is $T^{\gamma-2}$.
The other two regimes are treated in Appendix~\ref{app:acc-below}.
\end{proof}

With momentum the bias accumulates over the $T$ steps instead of being paid once.
\textbf{Any possible improvement through acceleration therefore vanishes as
$\gamma\nearrow2$:} the exponent of $\varepsilon$ rises by $\frac{\gamma-2}2$, that of
$d_0$ by $\gamma-2$, and the smoothness constant, absent from
Theorem~\ref{thm:main}(a), returns as $\beta^{(\gamma-2)/2}$. In the convex case,
acceleration degrades the bound exactly on the HTMC regime.

\subsection{The strongly convex case}

What meets the cost model is the number of oracle draws and not the horizon. Under strong
convexity we restart, rerunning the method from its own output whenever the expected
squared distance to the optimum has halved, as \citet{ghadimilan2012} and
\citet{dvurechensky2016} do. Conversion (b) needs no domain, so \eqref{eq:dg34} may be
applied again and again, each time from where the previous application stopped. Strong
convexity is what makes that pay: it turns a small loss into a small distance, so each
application starts closer than the last.

Let a stage be $N$ steps started from a point at expected squared distance $R_j^2$. With
$L=2\beta$ and $\delta_0\le2\delta^2/\mu$, and taking expectations over the start,
\eqref{eq:dg34} reads
\begin{equation}\label{eq:acc-sc}
\E\big[\mathscr{L}(y_{j+1})-\mathscr{L}^\ast\big]\ \lesssim\
\underbrace{\frac{\beta R_j^2}{N^2}}_{\text{accelerated}}
\ +\ \underbrace{\frac{N\delta^2}{\mu}}_{\text{accumulated bias}}
\ +\ \underbrace{\frac{\sigma R_j}{\sqrt N}}_{\text{not accelerable}} ,
\end{equation}
the middle term carrying $\delta^2$ and not $\delta$, strong convexity turning a gradient
bias into a loss quadratically. Its three terms are at most $\mu R_j^2/12$ each as soon as
\begin{equation}\label{eq:sc-cond}
N^2\ \gtrsim\ \kappa,
\qquad
\delta^2\ \lesssim\ \frac{\mu^2R_j^2}{N},
\qquad
\sigma^2\ \lesssim\ \mu^2R_j^2N .
\end{equation}
Their sum is then at most $\mu R_j^2/4$, and since strong convexity gives
$\norm{y-y^\ast}^2\le\frac2\mu(\mathscr{L}(y)-\mathscr{L}^\ast)$, the stage ends at
expected squared distance at most $R_j^2/2$. Chaining from $R_0=d_0$, the loss after the
$k$-th stage is at most $\mu R_{k-1}^2/4=\mu d_0^22^{-(k+1)}$, which is $\varepsilon$
after
\begin{equation}\label{eq:sc-horizon}
k\ \asymp\ \log\frac{\mu d_0^2}\varepsilon
\quad\text{stages, hence}\quad
T=kN\ \asymp\ \sqrt\kappa\,\log\frac{\mu d_0^2}\varepsilon \quad\text{steps.}
\end{equation}
Two features of \eqref{eq:sc-cond} fix the cost. The stage length $N\asymp\sqrt\kappa$
does not depend on $j$. And $\delta$ and $\sigma$ are bought once for the whole run while
$R_j$ shrinks, so both budgets bind at the last stage, where $R_j^2\asymp\varepsilon/\mu$:
\begin{equation}\label{eq:sc-budget}
\delta^2\ \asymp\ \frac{\mu\varepsilon}{\sqrt\kappa},
\qquad
\sigma^2\ \asymp\ \mu\varepsilon\sqrt\kappa .
\end{equation}

\begin{proof}[Proof of Proposition~\ref{prop:acc-sc}]
The compute is the number of calls, $T$ from \eqref{eq:sc-horizon}, times
$C_\gamma c^\gamma\delta^{2-\gamma}\sigma^{-2}$. With \eqref{eq:sc-budget},
\[
\E[\Cost]\ \asymp\
\sqrt\kappa\,\log\frac{\mu d_0^2}\varepsilon\cdot C_\gamma c^\gamma
\Big(\frac{\mu\varepsilon}{\sqrt\kappa}\Big)^{\frac{2-\gamma}2}
\frac1{\mu\varepsilon\sqrt\kappa}
\ =\ C_\gamma\Big(\frac c{\sqrt{\mu\varepsilon}}\Big)^{\!\gamma}
\kappa^{\frac{\gamma-2}4}\log\frac{\mu d_0^2}\varepsilon ,
\]
the exponent of $\kappa$ being $\frac12-\frac{2-\gamma}4-\frac12$ and that of
$\mu\varepsilon$ being $\frac{2-\gamma}2-1$. Dividing by the bound of
Theorem~\ref{thm:main}(b) leaves the announced multiplier, whose three factors are the
three budgets,
\[
\underbrace{\kappa^{-1/2}}_{\text{fewer steps}}\cdot
\underbrace{\kappa^{(\gamma-2)/4}}_{\text{sharper bias}}\cdot
\underbrace{\kappa^{1/2}}_{\text{tighter variance}}
\ =\ \kappa^{\frac{\gamma-2}4} :
\]
the horizon shrinks by $\kappa^{-1/2}$, the bias budget tightens by $\kappa^{1/4}$ and so
multiplies $\delta^{2-\gamma}$ by $\kappa^{(\gamma-2)/4}$, and the variance budget tightens
by $\kappa^{-1/2}$ and so multiplies $\sigma^{-2}$ by $\kappa^{1/2}$. That the same
multiplier holds in all three regimes is checked in Appendix~\ref{app:acc-below}.
\end{proof}

The exponent in $\varepsilon$ is $\gamma/2$ either way: acceleration does not change what
accuracy costs, and above $\gamma=2$ it costs a power of the conditioning. Unlike the
convex case the amplification is bounded here, a power of the conditioning rather than of
the horizon, so acceleration is not lost outright; what it could still win is the
logarithm, which Proposition~\ref{prop:schedule} already wins by scheduling the accuracies
instead.

\subsection{Below and at the threshold}
\label{app:acc-below}

The two proofs above balanced the channels against the HTMC branch of
Proposition~\ref{prop:oracle}. The other two branches are balanced identically, only the
per-call compute changes, and this completes Proposition~\ref{prop:acc-cvx}.

\paragraph{Convex.} With $T\asymp d_0\sqrt{\beta/\varepsilon}$,
$\delta\asymp\varepsilon^{3/2}/(d_0^2\sqrt\beta)$ and
$\sigma^2\asymp\varepsilon^{3/2}\sqrt\beta/d_0$ as before, the per-call compute
$(c/\sigma)^\gamma$ of the branch $\gamma<2$ gives
\[
\E[\Cost]\ \asymp\ c^\gamma d_0^{1+\gamma/2}\,\beta^{(2-\gamma)/4}\,
\varepsilon^{-(2+3\gamma)/4},
\]
which is $T^{(\gamma-2)/2}$ times the unaccelerated bound of the same regime, the
ETMC row of the table following Corollary~\ref{cor:eta}; that multiplier is smaller
than one below the threshold, so acceleration helps exactly where the compute is not
$\eta$-invariant. At $\gamma=0$ it reduces to $d_0\sqrt{\beta/\varepsilon}$, the
iteration count of the exact accelerated method, each of whose steps is free. At
$\gamma=2$ the branch $c^2\sigma^{-2}\log^2(\sigma/\delta)$ gives
$c^2d_0^2\varepsilon^{-2}\log^2(\beta d_0^2/\varepsilon)$, again the unaccelerated
bound up to a constant, the multiplier being $1$. The three multipliers
$T^{\gamma-2}$, $T^{(\gamma-2)/2}$ and $1$ do cross at $\gamma=2$; the bounds they
multiply do not, the $\log^2$ standing in, as in Appendix~\ref{app:etmc}, for a
constant that diverges on either side of the threshold.

\paragraph{Strongly convex.} With $T\asymp\sqrt\kappa\log$ and
$\sigma^2\asymp\mu\varepsilon\sqrt\kappa$, the branch $\gamma<2$ gives
\[
\E[\Cost]\ \asymp\ c^\gamma(\mu\varepsilon)^{-\gamma/2}\,\kappa^{(2-\gamma)/4}\log,
\]
against $c^\gamma(\mu\varepsilon)^{-\gamma/2}\kappa^{1-\gamma/2}\log$ unaccelerated,
a multiplier $\kappa^{(\gamma-2)/4}$: the same as above the threshold. The reason is
that $\sigma/\delta=\sqrt\kappa$ in both methods and in every regime, so the factor of
Proposition~\ref{prop:oracle} that depends on that ratio cancels in the comparison and
only the homogeneity of degree $-\gamma$ survives.

\section{The numerical experiment in detail}
\label{app:protocol}

Every choice below was made for a reason, and where the reason is a measurement the
measurement is reported. We split the explanations as follows, in three parts: what the instance is and how
its rate is read, how the two methods are run and compared, and how far the
measurement reaches.

Throughout, $A\in\R^{m\times n}$ has columns ordered by decreasing norm,
$\norm{A_i}=i^{-1/\varphi}$, $A_{\le k}$ is $A$ with every column past the $k$-th zeroed and
$R_k=A-A_{\le k}$ holds the discarded tail. The residual is $r(X)=AX-b$, the objective
$L(X)=\tfrac12\norm{r(X)}^2$ and $L_0=L(0)$. Level $k$ has the two truncations of Remark~\ref{rem:reading}:
$A_k(X)=A_{\le k}^\top(A_{\le k}X-b)$, which the runs use, at $4mk$ flops, and
$A_{\le k}^\top r(X)$, on which the rate is read, at $2mk$. Bias means
$\norm{\,\cdot\,-\nabla L}$ in both cases, and $\delta_k$ is that of $A_{\le k}^\top r$.
 The rate the design predicts is $\gamma=2\varphi/(2-\varphi)$, and
$\gamma_{\mathrm{env}}$ denotes the one measured on a given seed.

\begin{proof}[Proof of Proposition~\ref{prop:lsq}]

Write $R_k=A-A_{\le k}$. Adding and subtracting $A_{\le k}^\top r(X)$ splits the error in
two,
\[
\nabla L(X)-A_k(X)=R_k^\top r(X)+A_{\le k}^\top R_kX ,
\]
the discarded tail of the true gradient, and the cross term the restriction creates. Hence
\[
\norm{\nabla L-A_k}\ \le\ \norm{R_k}\big(\norm{r}+\norm{A_{\le k}}\norm X\big)\ \le\
\norm{R_k}\,M
\]
on the ball, and dropping the second term gives the statement for $A_{\le k}^\top r$.
Since the columns are ordered,
$\norm{R_k}\le(\sum_{i>k}\norm{A_i}^2)^{1/2}\asymp k^{1/2-1/\varphi}$, the sum converging
exactly for $\varphi<2$, and it is the only factor depending on $k$. An accuracy $\delta$
therefore needs $k\asymp(M/\delta)^{2\varphi/(2-\varphi)}$ features at cost $\Theta(mk)$,
and $\tfrac{2\varphi}{2-\varphi}>2$ iff $\varphi>1$.
\end{proof}

\subsection{The instance, and the rate it is designed to have}
\label{app:part-instance}

The rate is a property of the design and not of the optimizer, so everything in this
part is measured at fixed points, before any descent is run.

\subsubsection{Configuration, and the constraints that fix it}
\label{app:instance}

Section~\ref{sec:lsq} fixes $\varphi$, $m$, the ladder and $n$. The full configuration is

\begin{center}
\setlength{\tabcolsep}{4pt}
\begin{tabular}{llll}
\hline
$\varphi=1.2$, so $\gamma=3$ & $m=2048$ & $k_{\max}=512=m/4$ & ladder $\{1,8,64,512\}$\\
$\rho=0.6$ & $n=131072=256\,k_{\max}$ & seeds $0,\dots,9$ & $C\in\{1,2,4,8,16,32\}$\\
\hline
\end{tabular}
\end{center}

The experiment is run over 10 seeds, with $\gamma=3$. We now get more into detail as for why the other parameters are what they are.

\paragraph{Level ratio of $8$.} The ratio on columns is $2^\gamma$, which is exactly the
ratio for which the bias halves from one level to the next, that is, the $2^{-k}$
normalization of Assumption~\ref{ass:dyadic}. A finer ladder is legal but pays for it: the
variance of the telescope carries a factor $(1-r^{\gamma/2-1})^{-2}$ in the bias ratio $r$,
worth $12$ at a column ratio of $8$, $24$ at $4$ and $84$ at $2$ when $\gamma=3$. The number of levels of the multilevel oracle is then fixed by what a feasible $m$ allows,
and what makes an $m$ infeasible is compute: at $k_{\max}=m/4$, a single call to the finest level already costs $m^2/2$ flops, before any sweep and before any seed.

\paragraph{$k_{\max}$ well below $m$.} This protects the \emph{loss floor} of a level, the
squared distance from $b$ to the span of the first $k$ columns, which behaves as
$k^{-2/\gamma}(m-k)/m$. The first $k$ columns span a $k$-dimensional subspace of $\R^m$, so
at $k=m$ they span everything: level $k$ then fits $b$ exactly, its floor is zero, and
there is nothing left for a finer level to improve on. The floor leaves the power law well
before that, so $k_{\max}=m/4$ keeps the distortion under a quarter at the finest level.
The same formula returns in Appendix~\ref{app:limits}, where it sets the tightest target
the experiment may ask for.

\paragraph{$k_{\max}$ well below $n$.} This protects the \emph{bias} of a level, which is
the size of the tail it discards, $(\sum_{i>k}\norm{A_i}^2)^{1/2}$. That sum runs over
every feature past $k$, so a design stopping at a finite $n$ is missing its far end: the finest
level suffers most, its bias comes out understated, and $\gamma$ with it. Here
$n=256\,k_{\max}$, and Appendix~\ref{app:limits} measures what the truncation still costs
at that ratio.

\paragraph{The sweep over $C$.} The theory fixes $C$ in closed form once $\delta$ and
$\sigma$ are chosen (Appendix~\ref{app:oracle}), but that value carries the conservative
constants of the proof: the triangle bound it rests on is measured at $0.88$ against the
$2$ it assumes. $C$ is therefore swept rather than computed, over a range that carries the
method from one always-evaluated level to three (Appendix~\ref{app:methods}).

That leaves $\rho$, the weight of a direction shared by every column. It is the one entry
of the table that is not a size, and without it the instance would not be in the HTMC
regime at all: Appendix~\ref{app:correlated} is why.

\subsubsection{Correlated columns, and why they are necessary}
\label{app:correlated}

 Assumption~\ref{ass:dyadic} constrains the envelope of the bias, the largest value each
level reaches over the domain. For $A_{\le k}^\top r$, the truncation the rate is read off, that envelope is worth measuring only if it is a single power law over the whole ladder. With independent columns it is not: the gradient at a point has two parts, and only one of them carries $\gamma$:
\[
\underbrace{X_j\norm{A_j}^2}_{\text{signal},\ \sim\,j^{-2/\varphi}}
\qquad\text{and}\qquad
\underbrace{\norm{A_j}\norm{b}/\sqrt m}_{\text{leakage},\ \sim\,j^{-1/\varphi}} .
\]
The slow term is the one that carries $\gamma$, and with independent Gaussian columns it is
damped by $\sqrt m$, because near-orthogonal columns barely see one another. The signal term
then dominates up to $j_\ast=(\sqrt m/\norm b)^{\varphi}$, and below that crossing the
measured exponent is $2\varphi/(4-\varphi)$, which never exceeds $2$ for any admissible
$\varphi$: $2\varphi/(4-\varphi)<2$ is exactly $\varphi<2$, and $\varphi<2$ is already
forced by $\sum_i\norm{A_i}^2<\infty$. The coarse end of the ladder would not be HTMC at
all. The usable range above $j_\ast$ grows only as
$m^{0.4}$,\footnote{$\norm b=\Theta(1)$ as $m$ grows, since
$\norm b^2=\sum_i\norm{A_i}^2X^{\ast\,2}_i$ converges for $\varphi<2$, so
$j_\ast\propto m^{\varphi/2}$ while the top of the ladder is $k_{\max}\propto m$. The usable
range is their ratio, $m^{1-\varphi/2}$, which is $m^{0.4}$ at $\varphi=1.2$: multiplying
$m$ by a thousand buys a factor $16$.} so no feasible $m$ repairs it.

A direction shared by every column removes the damping and drops the crossing to $O(1)$.
Measured at $m=2048$, fitting over the whole ladder:
\begin{center}
\setlength{\tabcolsep}{9pt}
\begin{tabular}{lcccc}
\hline
$\rho$ & $0.0$ & $0.2$ & $0.4$ & $0.6$\\
$\gamma$ measured & $1.45$ & $1.93$ & $2.77$ & $3.00$\\
\hline
\end{tabular}
\end{center}
The transition is smooth rather than a knife edge, and correlated features are more
realistic than an isotropic design. We take $\rho=0.6$, where the fit has $R^2=0.995$ over
the whole ladder. To be precise, $\rho$ does not change $\gamma$, which is fixed by
$\varphi$ alone: the columns are rescaled so that $\norm{A_i}=i^{-1/\varphi}$ exactly
whatever $\rho$ is. What $\rho$ moves is the crossing $j_\ast$ below which the signal term
dominates and the profile follows $2\varphi/(4-\varphi)$ instead. A fit over the whole
ladder then averages two regimes, which is what the low entries report. $\rho=0.6$ is not
the value that returns $3$, it is the value at which $j_\ast$ falls below the ladder and a single regime is left to measure.

\subsection{The two truncations, and which one measures the rate}

\subsubsection{What happens with an exact residual}
\label{app:residual}

One choice remains before the rate read in Appendix~\ref{app:reading} is pinned down: which of two
truncations the levels are. Both are available at level $k$, and the choice is not cosmetic:
\[
\text{(a)}\quad A_{\le k}^\top\big(A_{\le k}X_{:k}-b\big)=A_{\le k}^\top\big(A_{\le k}X-b\big),
\qquad\qquad
\text{(b)}\quad \big[A^\top(AX-b)\big]_{:k}=A_{\le k}^\top\big(AX-b\big) .
\]
Written this way the two differ in one symbol, $A_{\le k}$ against $A$ inside the residual,
and they agree while the iterate lives in the first $k$ coordinates. Once it carries mass
beyond $k$, coming from a finer draw at an earlier step, (a) rebuilds a residual that
pretends the columns between $k$ and $k_{\max}$ were never used. The two then differ by the
\emph{cross term} $A_{\le k}^\top R_kX$, where $R_k=A-A_{\le k}$ holds the discarded
columns (and this term indeed cancels out at $k=k_{\max}) $.

Both forms are covered by Proposition~\ref{prop:lsq}. Form (b) has for error the discarded
tail $R_k^\top r$ of the exact gradient and nothing else while form (a) carries that tail plus
the cross term, and the resulting bound
$\norm{R_k}\big(\norm{A_{\le k}}\norm{X}+\norm{r}\big)$ leaves $\norm{R_k}$ as the only
factor depending on $k$. The rate is therefore the same for both, and only the prefactor $M$ changes.

What differs is whether the bound is attained. The cross term is not a power law but a
window $(k,k_{\max}]$ that closes: it dominates at the coarse end of the ladder and
vanishes at the finest level, where the iterate carries no mass beyond $k$ and the two
forms agree identically. The slack of (a) therefore drifts along the ladder, from $1.0$ to
$4.9$, against $1.00$ to $1.03$ for (b), and a drift of that size is enough to move the
fitted exponent: $0.97$ for (a) against $2.91$ for (b) and a closed form of $3$
(Figure~\ref{fig:residual}). Form (a) satisfies Assumption~\ref{ass:dyadic} and does not
measure it. It is also the most computationally expensive of the two, by up to $23$ at equal
accuracy and never less.

\begin{figure}[H]
\begin{minipage}[t]{0.49\textwidth}
\vspace{0pt}
\centering
\includegraphics[width=\textwidth]{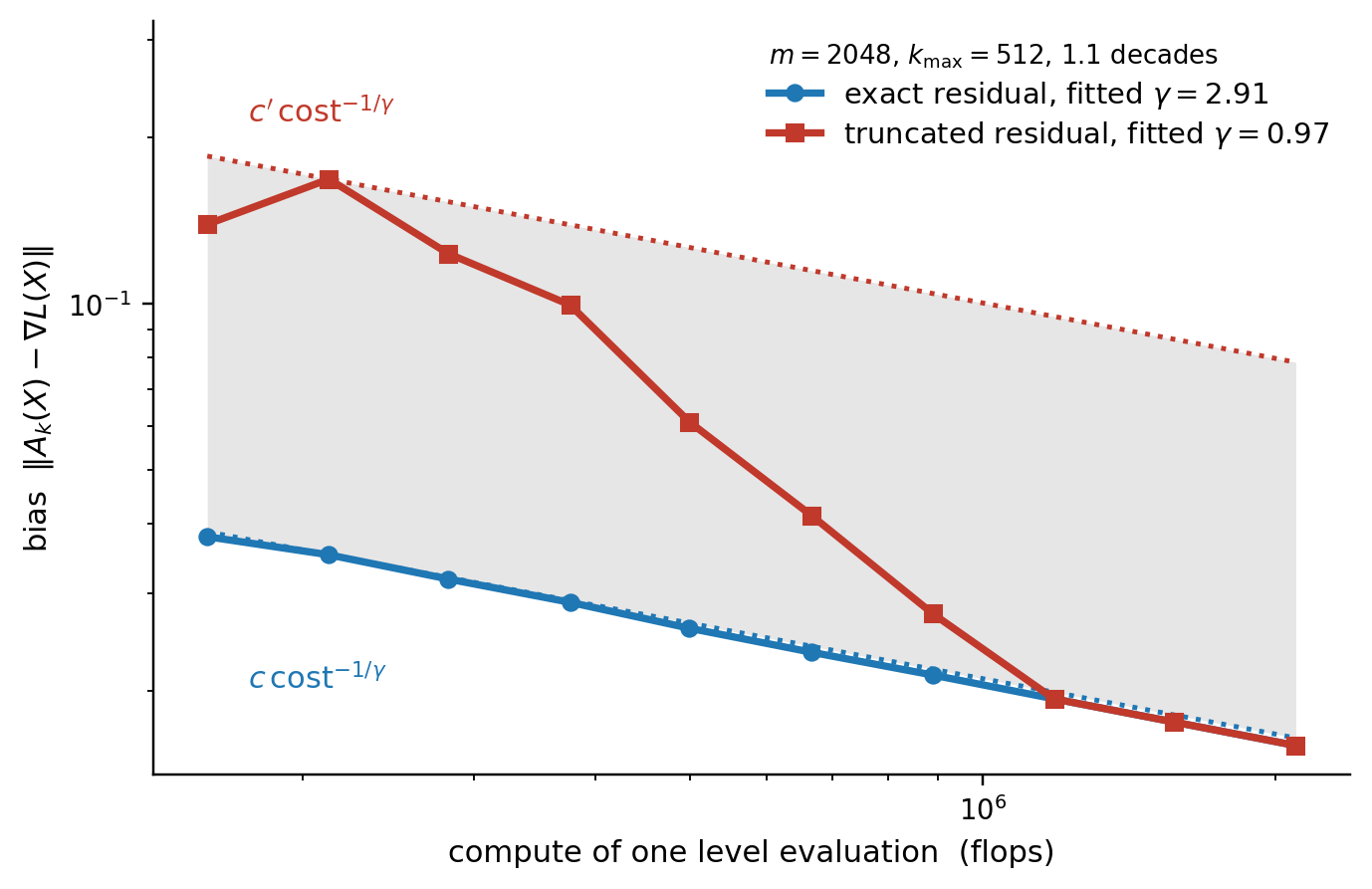}
\end{minipage}\hfill
\begin{minipage}[t]{0.49\textwidth}
\vspace{0pt}
\centering
\includegraphics[width=\textwidth]{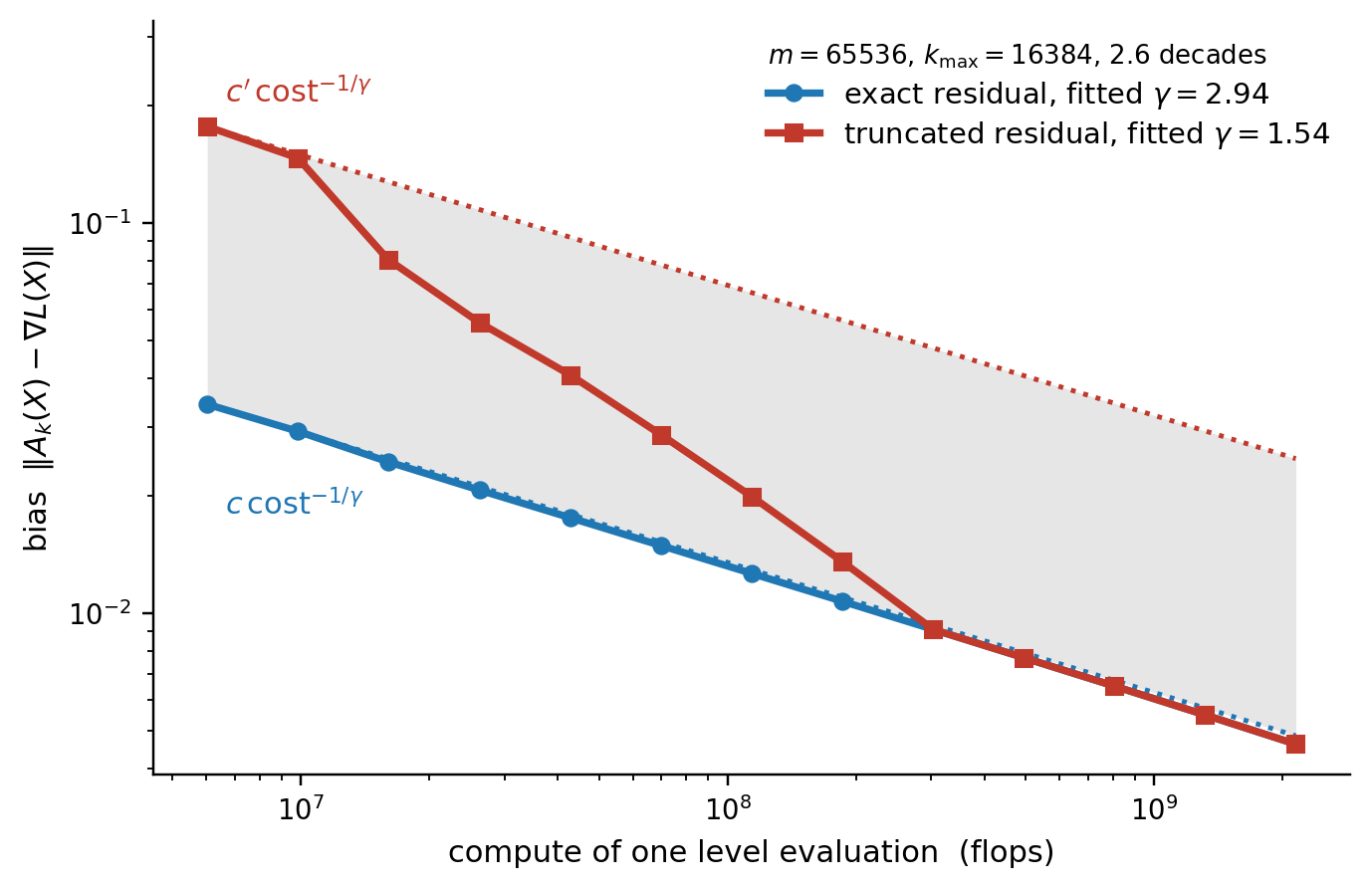}
\end{minipage}
\caption{The two truncations, bias against the compute of one level, at $m=2048$
(\textbf{left}) and $m=65536$ (\textbf{right}). The bias of a level is read as the largest
it reaches along a descent, which is the quantity Assumption~\ref{ass:dyadic} bounds. Each
dotted line is the envelope $c\,\mathrm{cost}^{-1/\gamma}$ the assumption grants its form and 
the rate is the same for both, so the two are parallel and the shaded band between them is
the factor by which $M$ differs, $4.8$ on the left and $5.1$ on the right.}
\label{fig:residual}
\end{figure}

Form (b) lies on the lower envelope throughout, to within $3\%$ and $5\%$, and a fit through it returns
$2.91$ and $2.94$ against a design value of $3$. Form (a) starts on the upper envelope and
falls across the band to meet (b) at the finest level, the cross term being a window that
closes rather than a power law. A line through that shape reports $0.97$ on the left and $1.54$ on the right. Neither
estimates $\gamma$, badly or otherwise: form (a) has the same rate as form (b) by
Proposition~\ref{prop:lsq}, and a line fitted to a shape that is not a power law returns
nothing in particular and evidence can be found in the fact that the two readings disagree, since a rate does
not move with the scale. What fails is the reading, not the rate.
Form (b) is thus a choice of measuring instrument rather than a correction, and it costs
one thing: it needs the residual $r=AX-b$, which is maintained rather than recomputed.

\begin{lemma}[The residual is amortized]\label{lem:amortized}
Write $O_t$ for the compute of the oracle call at step $t$ and $M_t$ for that of the
residual update it forces. Then $M_t\le O_t$ at every step of either method, so that
\[
\sum_t\big(O_t+M_t\big)\ \le\ 2\sum_tO_t ,
\]
and Assumption~\ref{ass:dyadic} holds along the run with $c$ replaced by $2^{1/\gamma}c$.
\end{lemma}

\begin{proof}
Throughout this part $A_k$ denotes the exact truncation $A_{\le k}^\top r$, the only one
that carries a residual and therefore the only one the lemma concerns.
 Write $\tilde f_t$ for the vector the oracle returns at step $t$ and $k_t$ for its support, $k_t$ thus being the finest level on which the oracle calls upon at step $t$. The step $X_{t+1}=X_t-\eta\tilde f_t$ moves the
residual by
\[
r(X_{t+1})\ =\ A\big(X_t-\eta\tilde f_t\big)-b\ =\ r(X_t)-\eta\,A\tilde f_t
=\ r(X_t)-\eta\,A_{\le k_t}\big[\tilde f_t\big]_{:k_t},
\]
the columns past $k_t$ multiplying zeros. This is a product of shape $(m,k_t)$, so
$M_t=2mk_t$ and we now bound $k_t$ by $O_t/2m$.

\emph{Fixed level $k$.} Here $\tilde f_t=A_{\le k}^\top r(X_t)$, supported on the first
$k$ coordinates and computed at cost $O_t=2mk$. Hence $k_t=k$ and $M_t=O_t$.

\emph{Telescope.} Let $k^-$ be the predecessor of $k$ in the ladder, $k^-=0$ at the
coarsest level where $A_0=0$, and let $B_k\in\{0,1\}$ record whether the increment reaching
$k$ was drawn.

Then \eqref{eq:mlmc} reads
\[
\tilde f_t\ =\ \sum_k\frac{B_k}{p_k}\Big(A_k-A_{k^-}\Big)(X_t),
\qquad
O_t\ =\ 2m\sum_kB_k\,(k+k^-),
\]
each drawn increment paying for the two levels it spans. Since $A_k-A_{k^-}$ is supported
on the first $k$ coordinates, $k_t=\max\{k:B_k=1\}$, and the widths being nonnegative,
\[
k_t\ =\ \max_{k\,:\,B_k=1}k\ \le\ \sum_kB_k\,k\ \le\ \sum_kB_k\,(k+k^-)\ =\ \frac{O_t}{2m},
\]
for every realization of the $B_k$, and not merely in expectation.

Summing over $t$ gives the display. Each call obeys $\Cost(A_k)\le c^\gamma2^{\gamma k}$,
so the run obeys the same bound with $c^\gamma$ doubled, since
$2c^\gamma=\big(2^{1/\gamma}c\big)^\gamma$, and only $\sum_t\Cost$ enters
Theorem~\ref{thm:main} so the theorem applies verbatim to the run.
\end{proof}

Measured over both sweeps, the ratio $M_t/O_t$ never exceeds one: it is exactly one at
every fixed level, so $2^{1/\gamma}$ is attained and not merely admissible, and below one
under the telescope.

The lemma rests on both operations reading the same $m$ rows: a level-$k$ call costs $2mk$
and the update $2ms$, so the comparison reduces to $s\le k$. It fails as soon as the
oracle becomes cheaper than a full pass over the design. Suppose it estimated the gradient
from $m'\ll m$ of the rows instead of all of them: its call would then cost $2m'k$, but
the residual it moves still has $m$ entries and all of them change, so the update still
costs $2ms$. The ratio $M_t/O_t$ grows like $m/m'$, and the maintenance takes over. Form
(a) is then the only one available: it rebuilds its residual from the columns it looks at,
and is stateless where (b) presupposes a current one.

\subsubsection{Reading $\gamma$: the envelope, and where the fit starts}
\label{app:reading}

This discussion refers back to the right panel of Figure~\ref{fig:forms}, and what follows is the protocol behind
it. \\

The bias of a level is not a single number: $\norm{A_{\le k}^\top r(X)-\nabla L(X)}$ depends on the
point $X$ at which the level is called, and Assumption~\ref{ass:dyadic} bounds it over the
whole domain. We therefore read $\gamma$ off the largest bias each level reaches over six
probes: the origin, where every run starts, and five iterates of a diagnostic descent of
$400$ steps driven throughout by the finest level oracle, at steps $1$, $4$, $20$, $89$ and
$400$. They are iterates rather than random points so that a probe resembles what the
oracle meets.

Six probes suffice, which we checked rather than assumed: 16 probes move
$\gamma_{\mathrm{env}}$ by at most $0.0026$, against a standard error of $0.007$ across
seeds, while 2 move one seed by $0.13$. No single probe would do: the last one alone
returns between $1.38$ and $1.94$ across the ten seeds, where the envelope returns between
$2.92$ and $2.99$.

A level that keeps only $k$ of the $n$ columns discards almost the whole gradient when $k$
is small, so its bias is just $\norm{\nabla L}$ and barely moves as $k$ grows. The coarse
end of the ladder is therefore a plateau rather than a power law. Fitting through it mixes
the two regimes, and the damage differs from seed to seed: over the ten seeds of the run,
\begin{center}
\setlength{\tabcolsep}{9pt}
\begin{tabular}{lcccc}
\hline
& mean & spread & range & lowest $R^2$\\
\hline
fit from $k=1$ & $2.809$ & $\pm0.256$ & $[2.28,\,3.10]$ & $0.954$\\
fit from $k=32$ & $2.963$ & $\pm0.023$ & $[2.92,\,2.99]$ & $0.999$\\
\hline
\end{tabular}
\end{center}
Excluding the plateau divides the spread by eleven and brings every seed within one
percent of the closed form, where including it sends one seed as low as $2.28$. What had
looked like variation between instances was the plateau.

The fit therefore starts at $k=32$, which drops nine of the eighteen levels probed. Two
things keep that from being a tuned choice. The plateau ends at an absolute $k$ and not at
a fraction of $k_{\max}$, since what ends it is the number of columns kept and not the
size of the design, so a larger design widens the usable range instead of moving the cut.
And the cut is drawn: it is the vertical dashed line marked \emph{fit starts here} on the right panel of Figure~\ref{fig:forms}, with the excluded levels still plotted to its left.

\subsection{What is measured, and how}
\label{app:part-measure}

Only now does an optimizer appear. The two methods we compare share everything but the oracle, and
the frontier is the single object the comparison is read from.

\subsubsection{The two methods}
\label{app:methods}

Both run plain gradient descent with the same accounting, the same guards and the same
reporting. The only difference is the oracle, which is the comparison Theorem~\ref{thm:main} is about.

\paragraph{Flops.} Both methods are charged the same way, $2mk$ for a product of shape
$(m,k)$, whether it evaluates a level or moves the residual (Lemma~\ref{lem:amortized}),
so the convention cancels in any comparison between them. No diagnostic flop is charged to
either, and the loss is read off the maintained residual, so it costs nothing at all.

\paragraph{Sampling probabilities, and the clamp.} We take
$p_k=\min\{C(\delta_k/\delta_0)^{1+\gamma/2},1\}$, written against the \emph{measured}
biases rather than against the level index. On a dyadic ladder this is exactly
$\min\{C2^{-(1+\gamma/2)k},1\}$, and written this way it stays correct on any geometric
one. The clamp creates a deterministic prefix of always-evaluated levels, and sweeping $C$
sweeps its depth: one level at $C=1$ and $C=2$ and two from $C=4$ on. The finer levels never join it: the finest is drawn with probability $0.005$ at $C=1$ and
$0.15$ at $C=32$, so the method stays genuinely randomised across the whole sweep.

\paragraph{Step size.} Both methods take $\eta=1/(4\beta)$, the constraint of
Theorem~\ref{thm:main}, but not with the same $\beta$. The baseline at level $k$ never
leaves the first $k$ coordinates, so its smoothness is $\beta_k$, the top eigenvalue of
$A_{\le k}^\top A_{\le k}$, found by power iteration and recomputed for every level it is
run at. The multilevel iterate can move anywhere in the first $k_{\max}$ coordinates, the
telescope drawing a different level at every step, so it takes $\beta_{k_{\max}}$
throughout.

The asymmetry runs against us: $\beta_k$ grows with $k$, so the baseline takes the larger
step at every level and the multilevel method the smallest on the ladder, though only from
$1.09$ at $k=4$ to $1.19$ at $k=512$. Neither step is fixed at that value: each method is
run at both $\eta$ and $\eta/2$, and its frontier keeps whichever reached a target for
less.

\paragraph{Safeguards, and checks.} Both the averaged and the last iterate are recorded and
the frontier takes the better of the two, for both methods alike: the theorem bounds the
average, but averaging helps a noisy method and hurts a deterministic one, and the
comparison should not turn on that. Three guards run throughout, a flop budget, a cap on
the number of iterations, and a divergence check, since the telescope can draw a fine level
with a weight $1/p_k$ of order a thousand (a diverged trajectory is truncated and the
frontier keeps its best pass). Three properties the accounting takes for granted are then
measured rather than assumed.
\begin{center}
\small
\setlength{\tabcolsep}{10pt}
\begin{tabular}{lcc}
\hline
quantity & measured & against\\
\hline
bias of the oracle, relative & $1.6\cdot10^{-3}$ & $7\cdot10^{-3}$\\
increment over predecessor bias, $V_k/\delta_{k-1}$ & $0.88\pm0.02$ & $2$\\
drift of the maintained residual, relative & $6.9\cdot10^{-6}$ & ---\\
\hline
\end{tabular}
\end{center}

The first is the distance from the mean of $20\,000$ draws at a fixed point to the exact
gradient. Its reference is the Monte Carlo error those draws carry on their own: an
unbiased estimator would still miss by that much, so landing a factor four inside it
leaves no bias to detect. The second is the increment against the triangle bound
$\norm{g_k-g_{k-1}}\le\delta_{k-1}+\delta_k\le2\delta_{k-1}$ that the proof uses, and what
matters is that the ratio does not drift with $k$: the shape of $p_k$ then follows the
biases as assumed, and only its scale is off, which the sweep over $C$ absorbs. The third
compares the residual carried through $5000$ steps in single precision with one recomputed
from scratch and it has no reference because nothing should accumulate at all.

\subsubsection{The frontier, and the baseline it is compared to}
\label{app:frontier}

For each target loss, the cost reported is the smallest cost, over every trajectory of the
sweep, at which that loss was ever reached: a running minimum along each trajectory, then a
minimum across trajectories. This is the only honest way to compare two families without
favouring the one that happened to be better tuned. Across seeds we report the median over
the seeds that reached the target, when at least half of them did.

The baseline is swept over a finer ladder than the oracle, of ratio $2$ rather than
$2^\gamma$ and of 10 levels rather than 4. The oracle is confined to its levels because
its probabilities presuppose them, but nothing confines a practitioner choosing a
truncation level, which is why the comparison does not favour the multilevel method by confining its rival: the oracle's four
levels are among the baseline's ten, so the six extra ones can only lower what the
baseline costs at any target, never raise it.

\begin{table}[H]
\centering\small
\setlength{\tabcolsep}{6pt}
\begin{tabular}{ccccccc}
\hline
 &  &  & \multicolumn{2}{c}{exact residual} & \multicolumn{2}{c}{truncated residual}\\
\cline{4-5}\cline{6-7}
seed & $\gamma_{\mathrm{env}}$ & fixed & multilevel & separation & multilevel & separation\\
\hline
$0$ & $2.923$ & $3.357$ & $2.400$ & $0.956$ & $3.172$ & $0.185$\\
$1$ & $2.977$ & $3.275$ & $2.133$ & $1.141$ & $2.484$ & $0.791$\\
$2$ & $2.969$ & $3.972$ & $2.462$ & $1.511$ & $3.077$ & $0.896$\\
$3$ & $2.969$ & $3.522$ & $2.349$ & $1.173$ & $2.965$ & $0.557$\\
$4$ & $2.994$ & $3.663$ & $2.298$ & $1.365$ & $2.796$ & $0.866$\\
$5$ & $2.977$ & $3.441$ & $2.185$ & $1.256$ & $2.858$ & $0.583$\\
$6$ & $2.927$ & $3.366$ & $2.084$ & $1.282$ & $2.807$ & $0.559$\\
$7$ & $2.950$ & $2.886$ & $1.614$ & $1.272$ & $2.258$ & $0.628$\\
$8$ & $2.979$ & $3.232$ & $1.746$ & $1.486$ & $2.074$ & $1.158$\\
$9$ & $2.961$ & $3.663$ & $2.566$ & $1.097$ & $3.065$ & $0.598$\\
\hline
mean & $2.963$ & $3.438$ & $2.184$ & $1.254$ & $2.756$ & $0.682$\\
standard error & $0.007$ & $0.093$ & $0.097$ & $0.054$ & $0.116$ & $0.083$\\
\hline
\end{tabular}
\caption{Per-seed exponents at $m=2048$ with the ten-level baseline ladder, fitted over the
same grid of relative targets on every seed. The fixed level is one column and not two: it
is the same object in both forms. Every fit has $R^2$ between $0.82$ and $1.00$. The
multilevel exponent is the smaller on all ten seeds in both forms and the gap exceeds 1 on
nine seeds with the exact residual and on one with the truncated.}
\label{tab:seeds}
\end{table}

\begin{table}[H]
\centering\small
\renewcommand{\arraystretch}{1.2}\setlength{\tabcolsep}{6pt}
\begin{tabular}{cccccc}
\hline
 & & \multicolumn{2}{c}{exact residual} & \multicolumn{2}{c}{truncated residual}\\
\cline{3-4}\cline{5-6}
$L/L_0$ & fixed & multilevel & ratio & multilevel & ratio\\
\hline
$5.01\cdot10^{-1}$ & $8.19\cdot10^{5}$  & $2.83\cdot10^{6}$ & $0.3\times$ & $3.06\cdot10^{6}$ & $0.3\times$\\
$2.70\cdot10^{-1}$ & $4.92\cdot10^{6}$  & $3.02\cdot10^{6}$ & $1.6\times$ & $4.43\cdot10^{6}$ & $1.1\times$\\
$1.46\cdot10^{-1}$ & $3.28\cdot10^{7}$  & $9.47\cdot10^{6}$ & $3.5\times$ & $2.14\cdot10^{7}$ & $1.5\times$\\
$7.85\cdot10^{-2}$ & $3.60\cdot10^{8}$  & $3.75\cdot10^{7}$ & $9.6\times$ & $1.61\cdot10^{8}$ & $2.2\times$\\
$4.23\cdot10^{-2}$ & $3.28\cdot10^{9}$  & $1.67\cdot10^{8}$ & $19.6\times$ & $8.07\cdot10^{8}$ & $4.1\times$\\
$2.28\cdot10^{-2}$ & $3.53\cdot10^{10}$ & $7.69\cdot10^{8}$ & $46.0\times$ & $3.25\cdot10^{9}$ & $10.9\times$\\
\hline
\end{tabular}
\caption{Compute in flops to reach a target loss, median over ten seeds: the numbers behind Figure~\ref{fig:forms}. At the loosest target
the telescope pays its overhead for nothing in both forms, and both cross near
$L/L_0=0.3$.}
\label{tab:costs}
\end{table}

\begin{remark}[The measured rate is not the intrinsic one]\label{rem:gammamin}
The truncation hierarchy obeys a scaling law of rate $\gamma=2\varphi/(2-\varphi)$ over the
range $k\le n$, and that is what Section~\ref{sec:lsq} measures. It is not the intrinsic
rate: $\nabla L(X)=A^\top(AX-b)$ is affine and computable exactly at a cost independent of
the accuracy, so $C(\nabla L,\varepsilon)$ is bounded and $\norm{\nabla L}_{M^\gamma}<\infty$
for every $\gamma>0$, whence $\gamma_{\min}(\nabla L)=0$. The instance is a measuring device
for the scaling law of a given family, not an example of an intrinsically expensive gradient.
\end{remark}

\subsection{The reach of the measurement}
\label{app:part-reach}

\subsubsection{Where the measurement stops}
\label{app:limits}

Four things bound the range over which the exponents above may be read, and each is
quantified rather than asserted.

\paragraph{The loss floor.} The iterates never leave the first $k_{\max}$ coordinates, where
the problem is overdetermined, so its minimum $L^\ast$ is not zero even though the full
problem interpolates. The formula of Appendix~\ref{app:instance} gives
$k_{\max}^{-2/\gamma}(m-k_{\max})/m=1.2\%$ of $L_0$, and solving the restricted problem
exactly returns $0.9\%$ on average over the ten seeds, ranging from $0.3\%$ to $2.0\%$. Since the loss reported is $L$ and not $L-L^\ast$, the targets must stay above that floor:
the dashed line on the left panel of Figure~\ref{fig:forms} is the largest of the ten, and the sixteen targets
above it are within reach of every seed where the four below are not.

\paragraph{The column crossing.} Part (b) of Theorem~\ref{thm:main} improves on part (a)
only for $\varepsilon\le\mu d_0^2$, and below that the predicted slope is bounded by $\gamma/2$ for
both methods rather than $\gamma+1$ and $\gamma$, so a single line fitted across it would
mix two regimes. This is why the targets stop above $\mu d_0^2$ (the dashed line on
Figure~\ref{fig:forms}, left), which the loss floor coincides with on this design: what is
read off there is governed by part (a) of Theorem~\ref{thm:main} alone.

\paragraph{The finite design.} Truncating the design at a finite $n$ understates the bias of
the finest level by a relative $(n/k_{\max})^{1-2/\varphi}$, which steepens the fitted line
and lowers $\gamma$. The prediction is testable, and it holds (three seeds):
\begin{center}
\setlength{\tabcolsep}{9pt}
\begin{tabular}{lccc}
\hline
$n/k_{\max}$ & $64$ & $256$ & $1024$\\
$\gamma$ measured & $2.888$ & $2.945$ & $2.968$\\
gap to $3.00$ & $-3.7\%$ & $-1.8\%$ & $-1.1\%$\\
\hline
\end{tabular}
\end{center}
The measured exponent is therefore not merely close to the closed form: the residual
deviation follows the predicted law and shrinks as its cause is removed. The run uses
$n=256\,k_{\max}$, with this table as the check.

\paragraph{The budget.} A fixed-level exponent measured far below $\gamma+1$ can also mean
the frontier ran out of budget. At the tightest target the analysis asks for some
$4\cdot10^5$ steps where the run takes $2.5\cdot10^5$. That count is an upper bound and
every seed still reaches at least sixteen of the twenty targets, so it is loose here, but
the margin is thin and the check at one fifth of the budget, which returns $3.31$, shows
the exponent still moving with it. Both causes lower the exponent, so a run must be
budgeted before its baseline ladder can be trusted.

\subsubsection{What is established, and what is not}
\label{app:verdict}

\paragraph{What is established.} The exponent of the hierarchy is measured at
$\gamma_{\mathrm{env}}=2.963\pm0.007$ against a closed-form $3$, with $R^2\ge0.999$ on every
seed. Moreover, the multilevel method has a strictly smaller cost exponent than the
fixed level on all ten seeds and in both truncations, by $0.682\pm0.083$ with the one the
runs use and $1.254\pm0.054$ with the one the rate is read off. That ordering, and not
either exponent on its own, is what the comparison is for.

\paragraph{What is not.} The absolute value of either exponent cannot be measured exactly. We measure
$3.438\pm0.093$ where the theorem allows $3.963$, and $2.184\pm0.097$ where it allows
$2.963$: deficits of $0.52$ and $0.78$, or $5.6$ and $8.1$ standard errors.

Remark~\ref{rem:local} accounts for them. The bias of a level factorises as
$\psi(X)\,k^{-1/\gamma}$, the power belonging to the design and the prefactor
$\psi(X)=\norm{r(X)}=\sqrt{2L(X)}$ to the point at which the level is called.
Assumption~\ref{ass:dyadic} needs that prefactor constant, so it takes the largest value
on the ball. A descent does not stay where it is largest, and what a run pays is an
average of $\psi$ along its trajectory.

Being a square root of the loss, $\psi$ turns that saving into a power of the target. A
power moves a slope without bending the line, which is why the frontiers stay straight and
only their exponents fall, and the exponent $1/2$ caps how far they can fall, at
$\gamma/2=1.48$. Both methods call the same oracle, so both collect the saving. What
separates them is the iteration count, which belongs to the optimizer and not to the
oracle.

\begin{remark}[The prefactor is local]\label{rem:local}
The residual is affine, $r(X)=A(X-X^\ast)+r^\ast$ with $r^\ast=AX^\ast-b$, so the bound
$\norm{\nabla L-A_{\le k}^\top r}\le\norm{R_k}\norm{r(X)}$ holds at every $X$, with a
prefactor $\norm{A}\norm{X-X^\ast}+\norm{r^\ast}$ affine in the distance to the optimum.
The ball of Proposition~\ref{prop:lsq} is only where that prefactor is frozen into a
constant, and freezing it costs nothing: Theorem~\ref{thm:main} keeps the iterates within
$O(d_0)$ of $X^\ast$, so $R$ may be taken there and $M\asymp\norm{A}d_0+\norm{r^\ast}$.
Carried through the proof, the local form replaces $c^\gamma$ by its average along the
trajectory rather than its supremum.
\end{remark}

\paragraph{And what the bound itself would look like.} No curve in Figure~\ref{fig:forms} is a
theoretical one, because none can be drawn on those axes: with $\Lambda_\gamma=2.6\cdot
10^{4}$ and $M$ frozen at the radius of the ball, Theorem~\ref{thm:main} sits
orders of magnitude above the compute actually spent. While the experiment tests an exponent and
a separation, it does not test a constant and it is not offered as doing so.

\end{document}

%% file: math_commands.tex
\usepackage{amsmath,amsfonts,bm}

\def\eqref#1{equation~\ref{#1}}
\def\1{\bm{1}}

\DeclareMathAlphabet{\mathsfit}{\encodingdefault}{\sfdefault}{m}{sl}
\SetMathAlphabet{\mathsfit}{bold}{\encodingdefault}{\sfdefault}{bx}{n}

\newcommand{\E}{\mathbb{E}}

\newcommand{\R}{\mathbb{R}}

\newcommand{\Var}{\mathrm{Var}}